\documentclass[11pt,reqno]{amsart} 
\usepackage{graphicx} 
\usepackage{amsfonts,amsmath,amssymb,amsthm}
\usepackage{biblatex}
\usepackage{xcolor}
\usepackage{color}
\usepackage{thmtools}
\usepackage{hyperref}

\newtheorem{prop}{Proposition}
\newtheorem{lema}{Lemma}
\newtheorem{rmk}{Remark}

\newtheorem{thm}{Theorem}
\newtheorem*{thmA}{Theorem A}
\newtheorem*{thmB}{Theorem B}
\newtheorem*{thmC}{Theorem C}

\renewcommand{\phi}{\varphi}

\newcommand{\dvg}{\text{d}v_g}

\newcommand{\dvl}{\text{d}a_g}

\renewcommand{\d}{\text{d}}
\newcommand{\Br}{\mathbb{B}^n(r)}

\newcommand{\R}{\mathbb{R}}
\newcommand{\Sn}{\mathbb{S}^n}

\renewcommand{\SS}{\mathbb{S}}

\newcommand{\canSn}{g_{can \SS^n}}

\allowdisplaybreaks

\title{Free boundary minimal Möbius band  in spherical caps}
\author{Mateus Spezia}
\address{IMPA - Instituto de Matematica Pura e Aplicada, Rio de Janeiro, RJ, Brasil.}
\email{mateus.spezia@impa.br}

\begin{document}
\begin{abstract}
We study compactly free boundary minimal submanifolds in spherical caps $\Br$ and their geometric spectral properties. Following the foundational work of Fraser-Schoen \cite{FS2012}, Lima-Menezes \cite{LM23} established the connection between free boundary minimal surfaces in spherical caps and spectral geometry. In this work, we present three main contributions: (1) We prove that any free boundary minimal Möbius band in $\Br$ immersed by first Steklov eigenfunctions must be intrinsically rotationally symmetric ; (2) We explicitly construct such a Möbius band in $\mathbb{B}^4(r)$ for $0<r<\frac{\pi}{2}$; and (3) We generalize Morse index estimates for free boundary minimal submanifolds in spherical caps, showing that non totally geodesic immersions have index at least $n$.
\end{abstract}
\maketitle
\section{Introduction}

The relation between minimal submanifolds in simply connected constant sectional curvature spaces and spectral problems has been studied for a long time. Takahashi \cite{Takahashi66} proved that a  submanifold $M^k$ of the unit Euclidean sphere $\mathbb{S}^n\subset \mathbb{R}^{n+1}$ is minimal if and only if its coordinate functions  $x_i=\langle p,e_i\rangle$  (where $p\in M$  and $\{e_i\}_{0\le i\le n}$ is the canonical basis of $\R^{n+1}$) are eigenfunctions, corresponding to the eigenvalue $k$, of the Laplace-Beltrami operator of the induced metric on $M$. 

In  a foundational work, Fraser and Schoen \cite{FS2011,FS2012} established an analogous connection for free boundary submanifolds in Euclidean balls. Here, \emph{free boundary} means that the submanifold $M \subset \mathbb{B}^n$ meets the boundary of the ball $\partial \mathbb{B}^n$ orthogonally, i.e., the outward unit conormal $\nu$ of $\partial M$ coincides with the position vector $x$, i.e., the ambient normal to $\partial \mathbb{B}^n$. They observed that such a submanifold $M$ is minimal and free boundary if and only if  its coordinate functions  satisfy
\[
\Delta_g x_i = 0 \quad \text{in } M, \quad \frac{\partial x_i}{\partial \nu} = x_i \quad \text{on } \partial M.
\]

Recently, Lima and Menezes \cite{LM23} considered a  new spectral problem related to free boundary minimal surfaces in \textit{spherical caps}. This problem  also was studied by Medvedev\cite{medvedev2025}.  Let \( \mathbb{S}^n \subset \mathbb{R}^{n+1} \) be the unit sphere, with \( e_0 = (1, 0, \dots, 0), \dots, e_n = (0, \dots, 1) \) denoting the canonical basis. For \( 0 < r < \pi \), the \textit{spherical cap} \( \Br \) is the geodesic ball
\[
\Br = \{ x = (x_0, x_1, \dots, x_n) \in \mathbb{S}^n \mid x_0 \ge \cos r \}.
\]

Inspirated by the Takahashi's result and the fact that the unit normal to the spherical cap $\partial\Br$ is given by
\begin{equation}\label{eq:Unity_normal_to_spherical_cap}
    N_{\partial\Br}(p)=\frac{1}{\sin r}(\cos r p-e_o)
\end{equation}
Lima and Menezes prove that if $\Phi: (\Sigma^k, g)\to \Br$ is a free boundary minimal immersion, then the coordinate functions $\phi_i=x_i\circ\Phi$ satisfy

\begin{equation}\label{eq:minimal_condition}
\Delta_g \varphi_i + k \varphi_i = 0, \quad \text{in } \Sigma, \quad i = 0, 1, \ldots, n, 
\end{equation}
\begin{equation}\label{eq:free_boundary_condition}
    \begin{split}
        \frac{\partial \varphi_0}{\partial \nu} + (\tan r)\varphi_0 &= 0, \quad \text{on } \partial \Sigma,  \\
\frac{\partial \varphi_i}{\partial \nu} - (\cot r)\varphi_i &= 0, \quad \text{on } \partial \Sigma, \quad i = 1, \ldots, n. 
\end{split}
\end{equation}

\noindent where $\nu$ is the outward
pointing unit $g$-normal vector of $\partial \Sigma$.

Lima-Menezes, following the ideas of  the work of Fraser and Schoen, established a connection between free boundary minimal surfaces and a Steklov-type eigenvalue problem. The \textit{Steklov eigenvalue problem} involves the following construction: Given a compact surface $(\Sigma, g)$ with boundary and a frequency parameter $\alpha$, where $\alpha$ is not a Dirichlet eigenvalue of $\Delta_g$, for any boundary function $u \in C^\infty(\partial \Sigma)$ there exists a unique extension $\hat{u} \in C^\infty(\Sigma)$ solving the boundary value problem:
\[
\Delta_g \hat{u} + \alpha \hat{u} = 0 \quad \text{with} \quad \hat{u}|_{\partial \Sigma} = u.
\]
The corresponding \textit{Dirichlet-to-Neumann operator} $\mathcal{D}_\alpha$ is then defined as
\[
\mathcal{D}_\alpha(u) := \frac{\partial \hat{u}}{\partial \nu},
\]
where $\nu$ is the outward-pointing unit conormal to $\partial \Sigma$.

A \( \sigma \in \mathbb{R} \) is called a \textit{Steklov eigenvalue with frequency \( \alpha \)} if
\[
\mathcal{D}_\alpha u = \frac{\partial \hat{u}}{\partial \nu} = \sigma u.
\]
The spectrum of \( \mathcal{D}_\alpha \) is discrete and satisfies
\[
\sigma_0(g, \alpha) < \sigma_1(g, \alpha) \leq \sigma_2(g, \alpha) \leq \cdots \to \infty.
\]
See the work \cite{ArendtMazzeo} for more details on the Steklov problem with frequency.

The case of surfaces in this setting is particularly interesting, as free boundary minimal surfaces correspond to metrics that extremize certain Steklov eigenvalue functionals. Fraser and Schoen studied the Steklov eigenvalue problem with frequency $\alpha = 0$. For a compact surface $(\Sigma, g)$ with boundary, they considered the first nontrivial Steklov eigenvalue $\sigma_1(g,0)$ and its normalized version \( \overline{\sigma}_1(g) = \sigma_1(g) |\partial \Sigma|_g,\)
where $|\partial \Sigma|_g$ denotes the boundary length. Their fundamental results show that:

\begin{itemize}
\item Any metric $(M, g)$ achieving $\sigma_1^*(M) = \sup_g \overline{\sigma}_1(g)$ arises from a free boundary minimal immersion into some Euclidean ball, \textbf{Theorem 1.1} in \cite{FS2012} 
\item The supremum $\sigma_1^*(M)$ is finite, \textbf{Theorem 2.2}  in \cite{FS2012}, and attained in the case of anulus and the Möbius band, see \textbf{Theorem 1.2}  and \textbf{Theorem 1.4}  in \cite{FS2012}. 
\end{itemize}

Subsequent work by Petrides \cite{Petrides} and Karpukhin-Kusner-McGrath-Stern \cite{KKMS} established the existence of maximizing metrics for all surfaces with boundary of genus zero or one. Currently, the only explicitly known maximizing metrics are those induced by the critical catenoid and the critical Möbius band \cite{FS2012}. For comprehensive surveys of these results, we refer to \cite{Lisurvey} and \cite{surveySteklov}.

Lima and Menezes were particularly interested in Steklov eigenvalues with frequency \( 2 \). For simplicity, we refer to them simply as \textit{Steklov eigenvalues} and denote \( \sigma_i(g, 2) \). They  introduced a new functional on the space \(\mathcal{M}(\Sigma) = \{ g \in \text{Riem}(\Sigma) \mid \lambda_1(g) > 2 \}\),
defined by
\[
\Theta_r(\Sigma, g) = \left[ \sigma_0(g,2) \cos^2 r + \sigma_1(g,2) \sin^2 r \right] |\partial \Sigma|_g + 2 |\Sigma|_g,
\]
and proved that:
\begin{itemize}
    \item The supremum of $\Theta_r(\Sigma, g)$ is  always finite, see \textbf{Theorem A} in  \cite{LM23}, and  maximizing metrics for \( \Theta_r(\Sigma, g) \) are induced by free boundary minimal immersions into spherical caps \(\Br\), where the immersion is given by first Steklov eigenfunctions, see \textbf{Theorem B} in \cite{LM23}
    \item They characterized  immersion given by first Steklov eigenfunction in the case of annuli, see \textbf{Theorem C} in \cite{LM23}.
\end{itemize}
Since this work deals exclusively with Steklov eigenvalues of frequency $\alpha=2$, we adopt the convenient terminology of referring to them simply as\textit{Steklov eigenvalues}, and writing $\sigma_k(g,2)=\sigma_k(g)$. Given an immersion 
\[
\Phi = (\phi_0, \phi_1, \dots, \phi_n) : (\Sigma, g) \to \mathbb{B}^n(r),
\]
we call $\Phi$ a \textbf{immersion by first Steklov eigenfunctions} if it satisfies:
\begin{itemize}
    \item $\phi_0$ is a $\sigma_0(g)$-Steklov eigenfunction
    \item $\phi_i$ are $\sigma_1(g)$-Steklov eigenfunctions for $i = 1, \dots, n$
\end{itemize}

In this work, or main interest is study free boundary Möbius band in spherical caps. For an study of the free boundary Möbius band in the Euclidean ball, see the work of Fraser and Sargent \cite{FraserSargent}. We focus on classifying free boundary minimal M\"obius bands in \( \Br \) immersed by first Steklov eigenfunctions, we prove that such M\"obius bands must be \textbf{intrinsically rotationally symmetric} in sense that, there exist local coordinates $(s,\theta)$ in which the Riemannian metric takes the conformal form $g=\rho(s,\theta)(ds^2+d\theta^2)$ and \( \rho \) depends only on \( s \), i.e. $\frac{\partial \rho}{\partial \theta} = 0$.

\begin{thmA}\label{thm:TheoremA}
        Let $0<r<\frac{\pi}{2}$, let  \((\Sigma, g)\) be a Möbius band, if  \(\Phi = (\phi_0, \ldots, \phi_n): (\Sigma, g) \to \mathbb{B}^n(r)\) is a free boundary minimal immersion by first Steklov eigenfunction. Then \(\Sigma\) is intrinsically rotationally symmetric. 
\end{thmA}
In particular, if we can prove that there exists a maximizing metric for \( \Theta_r(\Sigma, g) \) for the topology of the Möbius band, then it must be intrinsically rotationally symmetric. This result is analogous to the uniqueness of the critical Möbius band in the Euclidean ball (\textbf{Theorem 7.4} in \cite{FS2012}).

For each \( 0 < r < \frac{\pi}{2} \), we construct a intrinsically rotationally symmetric free boundary minimal Möbius band in the spherical cap \( \mathbb{B}^4(r) \).

\begin{thmB}
Let \( 0 < r < \pi/2 \). Then there exists a free boundary minimal immersion of a Möbius band by first Steklov eigenfunction
\[
\Phi = (\phi_0, \dots, \phi_4) : (\mathbb{M}, g) \to \mathbb{B}^4(r).
\]
\end{thmB}

The construction of these a Möbius band is inspired by \cite{Jakobson_Nadirashvili_Polterovich_2006} and \cite{ElSoufi_Giacomini_Jazar}, where the authors constructed a minimal Klein bottle in \( \mathbb{S}^4 \). An interesting geometric observation emerges from the proof of \textbf{Theorem B}: Half of the klein bottle construct in  \cite{ElSoufi_Giacomini_Jazar, Jakobson_Nadirashvili_Polterovich_2006} corresponds to a free boundary minimal Möbius band, as discussed in  \textbf{Remark \ref{rmk:r=pi/2}}. 

We also establish a result concerning the Morse index of free boundary minimal submanifolds in spherical caps \( B_r \), generalizing the codimension one case studied in \cite{LM23} and \cite{medvedev2025}. We prove: 
\begin{thmC}\label{thm:TheoremC}
    Let $r\neq  \frac{\pi}{2}$. Let \( \Phi: \Sigma^k \to \mathbb{B}^n(r) \) be a compact free boundary minimal immersion. If \( \Sigma \) is not contained in a hyperplane passing through the origin, then the Morse index of \( \Sigma \) is at least \( n \).
\end{thmC}

The paper is organized as follows. In Section \ref{sec: index estimate}, we prove \textbf{Theorem C}, which will be instrumental for the proof of \textbf{Theorem A} in Section \ref{sec:uniqueness}. The ODE analysis leading to the proof of \textbf{Theorem B} is performed in Section \ref{sec:ODE}. \\

\noindent \textbf{Acknowledgments:}  This work was developed as part of my thesis at IMPA under the supervision of Lucas Ambrozio, to whom I am deeply grateful for his guidance and support. I also thank Vanderson Lima for hosting my visit to UFRGS, where part of this research was conducted.  Also thanks to my colleagues Carlos Toro and Ivan Miranda for their interest and valuable discussions about this work. This research was supported by CAPES - Coordenação de Aperfeiçoamento de Pessoal de Nível Superior.

\section{Index estimatives}\label{sec: index estimate}

Let $\mathbb{B}^n(r)\subset \mathbb{S}^n$ be a spherical cap with center in $e_0=(1,0,\ldots,0)$ and radius $0<r<\pi/2$. Let \(\phi: (\Sigma^k, g) \to \mathbb{B}^n(r)\) be a free boundary minimal immersion. Our goal in this section is to show that, if the immersion is not totally geodesic, then there exists a $n$-dimensional space of normal deformations of $\phi$ that decrease area to second order.
    
    The index form associated with the second variation of the area functional is given by
\begin{equation}\label{eq:index_form}
\begin{split}
        \mathcal{I}(V, W)&= \int_\Sigma \langle \overline{\nabla} V,\overline{
        \nabla} W\rangle -(k+|B|^2)\langle V,W\rangle \dvg+\int_{\partial\Sigma}\langle \overline{\nabla}_V W, \nu \rangle \dvl, \\
        &= -\int_\Sigma \langle L_\Sigma V, W \rangle \, \dvg+ \int_{\partial\Sigma} \left( \langle \overline{\nabla}_\nu V, W \rangle + \langle \overline{\nabla}_V W, \nu \rangle \right) \dvl,
\end{split}
\end{equation}
where \(V, W \in \mathfrak{X}^\perp(\Sigma)\). In the above formulas, $\overline{\nabla}$ denotes the connection of $\mathbb{S}^n$, \(B(X,Y)=(\overline{\nabla}_{X}Y)^{\perp}\) denotes the second fundamental form of \(\Sigma\) in $\mathbb{S}^n$ on \(X, Y \in \mathfrak{X}(\Sigma)\), \(L_\Sigma = \Delta^\perp + (k+|B|^2)\) is the Jacobi operator of $\Sigma$, and $\nu$ is the outward-pointing unit conormal of $\partial \Sigma$ in $\Sigma$.

Let \( y \in \mathbb{R}^{n+1} \) be a fixed vector. Consider the tangent vector field on \( \mathbb{S}^n \) defined by
\[
\partial_y(p) = \operatorname{pr}_{T_p\mathbb{S}^n}(y) = y - \langle p, y \rangle p,
\]
where \(\operatorname{pr}_{T_p\mathbb{S}^n}\) denotes the orthogonal projection onto the tangent space \(T_p\mathbb{S}^n\). Let \(\partial_y^\perp\) denote the component of \(\partial_y\) normal to \(T_p\Sigma\) and \(\partial_y^\top\) denote the component of \(\partial_y\) tangent to \(T_p\Sigma\). We also define the function \(\phi_y(p) = \langle p, y \rangle\) on \(\Sigma\).

We then have the following useful identities:
\begin{equation}\label{eq: propriedades of u_y}
    \begin{split}
        \overline{\nabla} \phi_y &= \partial_y, \\
        \nabla \phi_y &= \partial_y^{\top}, \\
        \overline{\nabla}_X \partial_y &= -\phi_y X, \\
        \Delta_\Sigma \phi_y &= -k \phi_y,\\
        \langle \partial_y,\partial_x\rangle&= \langle y,x\rangle-\phi_y \phi_x,
\end{split}
\end{equation}
where \(X\) is a vector field in \(\Sn\) and   \(\overline{\nabla}\) denotes the gradient with respect to the ambient metric on \(\mathbb{S}^n\), and \(\nabla\) is the gradient with respect to the intrinsic metric \(g\) on \(\Sigma\). Notice that the fourth equation holds because the immersion of $(\Sigma,g)$ is minimal.

Let \( y \in \mathbb{R}^{n+1} \) be a vector orthogonal to \( e_0 \). Consider the normal vector field on $\Sigma$ defined by
\begin{equation}\label{eq:variation_field}
   V_y = \phi_y \, \partial_0^\perp - \phi_0 \, \partial_y^\perp + \frac{1 + \sin r}{\cos r}\, \partial_y^\perp \in \mathfrak{X}^\perp(\Sigma),
\end{equation}
where, for simplicity, we write \( \phi_0 = \phi_{e_0} \), \( \partial_0 = \partial_{e_0} \). 

Let \( V_1 = \phi_y\, \partial_0^\perp - \phi_0\, \partial_y^\perp \) be the normal component on $\Sigma$ of the  Killing vector field $K(p)=\langle y,p\rangle e_0-\langle e_0,p\rangle y$ on $\mathbb{S}^n$. If we define the constant \( c(r) = \frac{1 + \sin r}{\cos r} \), then we can write
\[
V_y(p) = V_1 + c(r)\, \partial_y^\perp,
\]

Observe that
\begin{equation}\label{boundary value of Vy}
V_y|_{\partial\Sigma} = \big(-\phi_0(p) + c(r)\big) \partial_y^\perp = \frac{\sin r(1+\sin r)}{\cos r} \partial_y^\perp ,
\end{equation}
because the free boundary condition with respect to $\mathbb{B}^n(r)$ gives us \( \partial_0^\perp|_{\partial\Sigma} = 0 \), and we also have \( \phi_0(p) = \cos r \) for \( p \in \partial\Sigma \).

\begin{lema}\label{Lemma: dimesion of the V}
The set of $\mathcal{V}=\{V_y\in \mathcal{X}^{\perp}(\Sigma)\,:\,\langle y,e_0\rangle = 0\}$ is real vector space of dimensions $n$, unless $\Sigma$ is contained in a hyperplane that contains $e_0$.
\end{lema}

\begin{proof}
The map $y\in span\{e_0\}^{\perp} \mapsto V_y\in \mathcal{X}^{\perp}(\Sigma)$ is linear. If its kernel is non-trivial, then the vector field $\phi_y\partial_0-\phi_0\partial_y+c(r)\partial_y$ is tangent to $\Sigma$, for some $y \neq 0$. By the free boundary condition, it follows that \( \partial_y^\perp|_{\partial \Sigma} = 0 \). Moreover, note that, for every $X,Y\in \mathcal{X}(\partial \Sigma)$,
\begin{align*}
   \operatorname{Hess}_{\partial \Sigma} \phi_y (X,Y) & = \operatorname{Hess}_{\mathbb{S}^n}\phi_y(X,Y) + \left\langle \overline{\nabla}\phi_y, B^{\partial \Sigma}(X,Y) \right\rangle \\
        & = -\phi_y g(X,Y) + \left\langle \partial_y^{\perp}, B^{\partial \Sigma}(X,Y) \right\rangle \\
        & = -\phi_yg(X, Y),
\end{align*}
where $B^{\partial\Sigma}$ is the second fundamental form of $\partial \Sigma$ in $\mathbb{S}^n$.
    
If $\phi_y$ does not vanish, then, by Obata's theorem \cite{Obata}, it follows from the above identity that each connected component of \( \partial \Sigma \) is isometric to a Euclidean sphere in $\partial \mathbb{B}^n(r)$ of radius one. This can only happen when then $r=\pi/2$, which is excluded from our considerations. Thus, $\phi_y$ vanishes on $\partial \Sigma$, that is, $\partial \Sigma$ is contained in the hyperplane orthogonal to $y$.  By the unique continuation principle  we conclude that the free boundary minimal submanifold \( \Sigma \) itself is contained in the hyperplane orthogonal to $y$.
\end{proof}

We proceed to compute the index form \eqref{eq:index_form} on the normal vector fields $V_y$.
Since \( K \) is Killing vector field, we have \( L_\Sigma V_1 = 0 \) (by Lemma 5.1.7 of \cite{Simons68}), so the Jacobi operator acts as:
\[
L_\Sigma V_y = c(r)\, L_\Sigma(\partial_y^\perp).
\]
Using Lemma 5.1.3 of \cite{Simons68}, we obtain:
\begin{equation}\label{eq:Jacobi_operator_on_V_y}
    L_\Sigma V_y = k\, c(r)\, \partial_y^\perp.
\end{equation}

 Using the equations in \eqref{eq: propriedades of u_y}, obeserve that
 \[ 
 \langle \partial_y^\perp,\partial_y^\perp\rangle=\langle \partial_y,\partial_y\rangle-\langle\partial_y^\top,\partial_y^\top\rangle=1-\phi_y^2-\langle\nabla \phi_y\,\nabla\phi_y\rangle,
 \]
\[ 
\langle \partial_y^\perp,\partial_0^\perp\rangle=\langle \partial_y,\partial_0\rangle-\langle\partial_y^\top,\partial_0^\top\rangle=-\phi_y\phi_0-\langle\nabla \phi_y\,\nabla\phi_0\rangle.
 \]
Now we can compute
\begin{equation}\label{eq: Jacobi interior}
\begin{split}
\langle L_\Sigma V_y, V_y \rangle 
&= \langle k\, c(r)\, \partial_y^\perp,\ \phi_y\, \partial_0^\perp - \phi_0\, \partial_y^\perp + c(r)\, \partial_y^\perp \rangle \\
&=k\, c(r)^2 \, \|\partial_y^\perp\|^2+ k\, c(r) \phi_y \langle \partial_y^\perp, \partial_0^\perp\rangle-k\, c(r) \phi_0 \langle \partial_y^\perp, \partial_y^\perp\rangle\\
&= k\, c(r)^2\, \|\partial_y^\perp\|^2 - k\, c(r)\, \phi_0 \\
&\quad + k\, c(r) \left( \phi_0\, \langle \nabla \phi_y, \nabla \phi_y \rangle - \phi_y\, \langle \nabla \phi_y, \nabla \phi_0 \rangle \right) \\
&= k\, c(r)^2\, \|\partial_y^\perp\|^2 + c(r)\, \Delta_\Sigma \phi_0 \\
&\quad + k\, c(r)\, \operatorname{div}_\Sigma\left( \phi_y\, \phi_0\, \nabla \phi_y - \phi_y^2\, \nabla \phi_0 \right).
\end{split}
\end{equation}

On the other hand, we can also compute the boundary terms in the index form \eqref{eq:index_form}. Since \( V_y|_{\partial\Sigma} \) is tangent to \( \partial\mathbb{B}^n(r) \), by \eqref{boundary value of Vy} we have
\begin{equation}\label{eqnablaVyVynu}
\langle \overline{\nabla}_{V_y} V_y, \nu \rangle = \langle B^{\partial\mathbb{B}^n}(V_y, V_y), \nu \rangle = -\cot r\, \|V_y\|^2 = -\frac{\sin r(1+\sin r)^2}{\cos r} \|\partial_y^\perp\|^2 ,
\end{equation}
where \( B^{\partial\mathbb{B}^n} \) the second fundamental form of \( \partial\mathbb{B}^n \) in \( \mathbb{S}^n \).

As \( \nu(p) = -\frac{1}{\sin r}\, \partial_0(p) \) for \( p \in \partial\Sigma \), and since \( \nu \) is tangent to \( \partial\Sigma \), we can take an orthonormal frame \( \{X_1, \ldots, X_k\} \) of \( T\Sigma \) such that \( X_k = \nu \) along \( \partial\Sigma \). Therefore, for \( i \neq k \),
\begin{equation*}\label{eq B(X,nu)=0}
    B(\nu, X_i) = B(X_i, \nu) = (\overline{\nabla}_{X_i} \nu)^\perp = \left( \frac{\phi_0}{\sin r} X_i \right)^\perp = 0,
\end{equation*}
where we used the third equation of \eqref{eq: propriedades of u_y}. It follows that
\begin{align*}
\begin{split}
    \langle \overline{\nabla}_\nu \partial_y^\perp, \partial_y^\perp \rangle 
    &= \langle \overline{\nabla}_\nu (\partial_y - \partial_y^\top), \partial_y^\perp \rangle \\
    &= \langle -\phi_y\nu - \sum_{i=1}^k \overline{\nabla}_\nu (\langle \partial_y, X_i \rangle X_i), \partial_y^\perp \rangle \\
    &=- \sum_{i=1}^k  \langle \partial_y, X_i \rangle\left\langle  B(\nu,X_i), \partial_y^\perp \right\rangle\\
    &= - \left\langle \partial_y, \nu \right\rangle \langle B(\nu, \nu),\partial_y^\perp \rangle \\
    &= -\cot{r}\cdot \phi_y  \langle B(\nu, \nu),\partial_y^\perp\rangle. \\
\end{split}
\end{align*}

By a similar computation, we obtain $\langle \overline{\nabla}_\nu \partial_0^\perp,\partial_y^\perp\rangle=\sin r \langle B(\nu,\nu),\partial_y^{\perp}\rangle$. We can then compute 
\begin{align}\label{eqnablanuVyVy}
\begin{split}
    \langle \overline{\nabla}_\nu V_y,V_y\rangle&=\langle \overline{\nabla}_\nu (\phi_y \partial_0^\perp+(c(r)-\phi_0)\partial_y^\perp, (c(r)-\cos{r})\partial_y^\perp\rangle \\
    &=\langle \phi_y \overline{\nabla}_\nu \partial_0^\perp-\frac{\partial\phi_0}{\partial\nu}\partial_y^\perp + (c(r)-\cos r)\overline{\nabla}_\nu \partial_y^\perp,(c(r)-\cos{r})\partial_y^\perp\rangle \\
    &= ( c(r)-\cos{r}) \Big( \sin{r}\cdot\phi_y \langle B(\nu,\nu),\partial_y^\perp\rangle+ \sin{r} ||\partial_y^\perp||^2\Big)\\ 
    & \quad -(c(r)-\cos{r})^2\cot{r}\cdot\phi_y \langle B(\nu,\nu) ,\partial_y^\perp\rangle \\
    &= \frac{\sin{r}(1+\sin{r})}{\cos{r}}\left(\sin r- \frac{\sin{r}(1+\sin r)}{\cos{r}}\cot r \right)\phi_y \langle B(\nu,\nu),\partial_y^\perp\rangle \\
    & \quad+\sin r\frac{\sin{r}(1+\sin{r})}{\cos{r}}||\partial_y^\perp||^2 \\
    &= -\frac{\sin{r}(1+\sin{r})}{\cos{r}}\phi_y \langle B(\nu,\nu),\partial_y^\perp\rangle+\frac{\sin^2{r}(1+\sin{r})}{\cos{r}}||\partial_y^\perp||^2.
\end{split}
\end{align}

Consider the vector field $\xi=\partial_y-\langle\partial_y,\nu\rangle\nu=\partial_y-\cot{r}\cdot \phi_y \nu $, which is tangent to $\partial B^n$ on points of $\partial \Sigma$. Observe that
\begin{equation}\label{eq: divente bordo I}
\begin{split}
        \text{div}_{\partial\Sigma} \xi&=\sum_{i=1}^{k-1} \langle \overline{\nabla}_{X_i} \xi,X_i\rangle  \\
        &=\sum_{i=1}^{k-1} \langle -\phi_y X_i- \cot{r}\cdot \phi_y \frac{\phi_0}{\sin{r}}X_i,X_i\rangle \\
        &=-(k-1) (1+ \cot^2{r})\phi_y=-\frac{(k-1)}{\sin^2{r}}\phi_y
\end{split}
\end{equation}
On the other hand, since $\xi^{\perp} = \partial_y^\perp$ on $\partial \Sigma$, we have
\begin{equation}\label{eq: divente bordo II}
\begin{split}
        \text{div}_{\partial\Sigma} \xi&=  \text{div}_{\partial\Sigma} \xi^{T\partial\Sigma}+\text{div}_{\partial\Sigma} \partial_y^\perp\\
        &=  \text{div}_{\partial\Sigma} \xi^{T\partial\Sigma}+\sum_{i=1}^{k-1}\langle\nabla_{X_i} \partial_y^\perp,X_i\rangle\\
         &=  \text{div}_{\partial\Sigma} \xi^{T\partial\Sigma}-\sum_{i=1}^{k-1}\langle\partial_y^\perp,B(X_i,X_i)\rangle\\
          &=  \text{div}_{\partial\Sigma} \xi^{T\partial\Sigma}+\langle\partial_y^\perp,B(\nu,\nu)\rangle,
\end{split}
\end{equation}
where in the last line we used that the immersion is minimal. Then, equations (\ref{eq: divente bordo I}) and (\ref{eq: divente bordo II}) give us that
\begin{equation}\label{eq: divergente bordo III}
    \phi_y \langle B(\nu,\nu),\partial_y^\perp\rangle=\|\partial_y^{\top\partial\Sigma}\|^2-\frac{(k-1)}{\sin^2{r}}\phi_y^2-\text{div}_{\partial\Sigma} (\phi_y \xi^{T\partial\Sigma}).
\end{equation}

Therefore, the boundary terms in the index form \eqref{eq:index_form} is given by 

\begin{equation}\label{eq:Bnd_terms_index_form_final_version}
\begin{split}
    \langle \overline{\nabla}_\nu V_y, V_y \rangle 
    + \langle \overline{\nabla}_{V_y} V_y, \nu \rangle
    &= -\frac{\sin r (1+\sin r)}{\cos r} \, \phi_y \langle B(\nu,\nu),\partial_y^\perp\rangle \\
    & \quad + \frac{\sin^2 r (1+\sin r)}{\cos r} \, \| \partial_y^\perp \|^2  \\
    &\quad - \frac{\sin r (1+\sin r)^2}{\cos r} \, \| \partial_y^\perp \|^2 \\
    &= -\frac{\sin r (1+\sin r)}{\cos r} \Bigg(
        -\frac{(k-1)}{\sin^2 r} \, \phi_y^2 
        + \| \partial_y^{\top \partial\Sigma} \|^2 
         \\
    &\qquad + \| \partial_y^\perp \|^2 - \operatorname{div}_{\partial\Sigma} (\phi_y \, \xi^{T\partial\Sigma}) 
    \Bigg) \\
    &= -\frac{\sin r (1+\sin r)}{\cos r} \Bigg(
        -\frac{(k-1)}{\sin^2 r} \, \phi_y^2 
         + \|\partial_y\|^2  \\
    &\qquad-\langle \partial_y,\nu\rangle^2 - \operatorname{div}_{\partial\Sigma} (\phi_y \, \xi^{T\partial\Sigma}) 
    \Bigg)\\
    &= -\frac{\sin r (1+\sin r)}{\cos r} \Bigg(
        -\frac{k}{\sin^2 r} \, \phi_y^2 + 1
        \\
    &\qquad - \operatorname{div}_{\partial\Sigma} (\phi_y \, \xi^{T\partial\Sigma}). 
     \Bigg)
\end{split}
\end{equation}

We then obtain the following result, which corresponds to \textbf{Theorem C}.

\begin{thm}\label{thm:index_estimate}
Let \( \Phi: \Sigma^k \to \mathbb{B}^n(r) \) be a free boundary minimal immersion, and let \( y \in \mathbb{R}^n \) be such that \( y \perp e_0 \). Then, for $V_y\in \mathcal{X}^{\perp}(\Sigma)$ as defined in \eqref{eq:variation_field},
\begin{equation}\label{eq:index_estimate}
    \mathcal{I}(V_y, V_y) = -k \left( \frac{1 + \sin r}{\cos r} \right)^2 \int_{\Sigma} \|\partial_y^\perp\|^2 \, \dvg.
\end{equation}
Moreover, if \( \Sigma \) is not contained in a hyperplane passing through the origin, then the Morse index of \( \Sigma \) is at least \( n \).
\end{thm}

\begin{proof}
    Using \eqref{eq: Jacobi interior} and \eqref{eq:Bnd_terms_index_form_final_version}, and integrating by parts, we have
    \begin{align*}
        \mathcal{I}(V_y, V_y) &= -k c(r)^2\int_\Sigma ||\partial_y^\perp||^2\dvg\\
        &\quad - \int_\Sigma c(r)\Big( k \operatorname{div}_\Sigma( \phi_y\, \phi_0\, \nabla \phi_y - \phi_y^2\, \nabla \phi_0)
        + \Delta_g \phi_0 \Big)\, \dvg\\
        &\quad -\frac{\sin r (1 + \sin r)}{\cos r} \int_{\partial\Sigma} \Bigg(1-\frac{k}{\sin^2 r}\phi_y^2 +
        \operatorname{div}_{\partial\Sigma} (\phi_y\, \xi^{T\partial\Sigma}) \Bigg) \,\dvl\\
        &= -k c(r)^2 \int_\Sigma ||\partial_y^\perp||^2 \, \dvg\\
        &\quad+ \int_{\partial\Sigma} \left( - k c(r) \phi_y \phi_0 \frac{\partial \phi_y}{\partial \nu} + k c(r) \phi_y^2 \frac{\partial \phi_0}{\partial \nu} - c(r)\frac{\partial \phi_0}{\partial \nu} \right) \,\dvl \\
        &\quad -\frac{\sin r (1 + \sin r)}{\cos r} \int_{\partial\Sigma} \left(1- \frac{k}{\sin^2 r} \phi_y^2  \right) \,\dvl \\
        &= -k c(r)^2 \int_\Sigma ||\partial_y^\perp||^2 \, \dvg \\
        &\quad+ \int_{\partial\Sigma} \left(  - k \frac{(1 + \sin r)\cos r}{\sin r} \phi_y^2 -k \frac{(1 + \sin r)\sin r}{\cos r} \phi_y^2 \right) \,\dvl \\
        &\quad + \int_{\partial\Sigma} \frac{\sin r (1 + \sin r)}{\cos r} \,\dvl\\
        &\quad+ \int_{\partial\Sigma} \frac{\sin r (1 + \sin r)}{\cos r} \left(\frac{k}{\sin^2 r } \phi_y^2 -1\right) \,\dvl \\
        &= -k c(r)^2 \int_\Sigma ||\partial_y^\perp||^2 \, \dvg.
    \end{align*}

    Finally, by \textbf{Lemma \ref{Lemma: dimesion of the V}}, we know that if \( \Sigma \) is not contained in a hyperplane passing through the origin, then the dimension of the space of normal sections generated by \( V_y \) is equal to \( n \). Since we proved that $I(V_y,V_y)<0$ for all $V_y\neq 0$, the estimate on the Morse index follows.
\end{proof}

\begin{rmk}
If $\Sigma$ lies  in a hyperplane through the origin, then its index equals $n-k$.
\end{rmk}

\section{Uniqueness of free boundary minimal Mobius Band in Spherical Caps}\label{sec:uniqueness}
    
We specialize the discussion of the previous Section to the case of surfaces.
Let $(\Sigma,g)$ be a compact  surface with boundary. Suppose that we have a free boundary minimal immersion $\Phi: \Sigma \to \mathbb{B}^n(r)$. We can decompose the tangent bundle as $\Phi^* T\mathbb{S}^n=\Phi_* T\Sigma\bigoplus^\perp T^\perp\Sigma$ , where $T^\perp\Sigma$ is the normal bundle. As before, we denote by $\overline{\nabla}$ the connection on $\mathbb{S}^n$ and by $\nabla$ the connection induced by immersion. In this section, it will be convenient to use $\langle -,-\rangle$ to denote the canonical metric on $\mathbb{S}^n$.

We define the energy of the map $\Phi$ as
\begin{equation}\label{eq: enegi funtional}
    E_g(\Phi)= \dfrac{1}{2}\int_\Sigma |d \Phi|^2 \dvg=  \frac{1}{2}  \sum_{i=1} ^2 \int |\d\Phi_p(E_i)|^2 \dvg,
\end{equation}
where $\{E_i\}_{i=1,2}$ can be taken as an arbitrary (local) orthonormal tangent frame on $(\Sigma,g)$. Notice that $E_g$ is invariant under conformal changes $g\rightarrow \rho g$, for any positive smooth function $\rho$ on the surface $\Sigma$. The map $\Phi$ is called \textit{harmonic} if is is a critical point of $E_g$

Given $V,W\in \Gamma(\Phi^* T \mathbb{S}^n)$ tangent to $\partial \mathbb{B}^n(r)$, that is, such that $V(p), W(p)\in T_{\Phi(p)} \partial \mathbb{B}^n (r)$ for all $p\in \Sigma$,  consider the billinear form
\begin{align}\label{eq: second variantion of energi0}
 \begin{split}   Q(V,W) & =\int_\Sigma \langle\nabla V, \nabla W\rangle- \sum_{i=1}^{2} R^{\mathbb{S}^n}(V,E_i,W,E_i) \dvg-\cot{r} \int_{\partial\Sigma} \langle V, W\rangle \dvl. \\
 & = \int_\Sigma \langle\nabla V, \nabla W\rangle - 2\langle V,W\rangle + \langle V^\top,W^\top\rangle \dvg-\cot{r} \int_{\partial\Sigma} \langle V, W\rangle \dvl. 
\end{split}
\end{align}
Here, $R^{\mathbb{S}^n}(X,Y,Z,W)=\langle X,Z\rangle\langle Y,W\rangle - \langle X,W\rangle\langle Y,Z\rangle$ denoted the curvature tensor of $\mathbb{S}^n$.

{The billinear form \eqref{eq: second variantion of energi0} is associated with the second variation of the energy functional for the harmonic map $\Phi$ along admissible variations. In fact, let $\Phi_t$ be an admissible variation of $\Phi$, that is, a one-parameter family of immersions $\Phi_t:\Sigma\rightarrow \mathbb{B}^n(r)$ such that $\Phi_0=\Phi$ and $\Phi_t(p)\in \partial \mathbb{B}^n(r)$ for every $p\in \partial \Sigma$ and $-\epsilon<t<\epsilon$. In this case, the variational vector field  $\dot{\Phi}=\frac{d}{dt}_{t=0}  \Phi_s=V$ satisfies $V(p)\in T_{\Phi(p)}\partial \mathbb{B}^n(r)$ for every $p\in \partial \Sigma$, and we have
 \begin{equation}\label{eq: second variation of energi}
     \frac{1}{2} \frac{d^2}{d t^2} |_{t=0} E_g(\Phi_t)=Q(V,V).
 \end{equation}

We call a section $V\in \Gamma(\Phi^* T\mathbb{S}^n)$ if, for any $X,Y \in \Gamma(T\Sigma)$, there exists a function $\xi$ such that 
\begin{equation}\label{eq: conformal section}
   \langle \nabla_X V,Y\rangle+\langle X,\nabla_Y V \rangle = 2 \xi \langle X, Y\rangle.
\end{equation}

 the next lemma, we relate the values of the quadratic forms $Q$ and $I$ on conformal vector fields that are tangent to $\partial \mathbb{B}^n(r)$.
 
\begin{lema}\label{lemma: Relacao da variacao de energia e da area}
Let $\Phi: (\Sigma,g)\to \mathbb{B}^n(r)$  be a free boundary minimal immersion of a compact surface with boundary. Suppose that $\Phi_t$ is an admissible variation of $\Phi$ such that $\dot{\Phi}=V \in \Gamma(\Phi^* T\mathbb{S}^n)$ is conformal.  Then
\begin{equation*}
    Q(V,V)=\mathcal{I}(V^\perp,V^\perp). 
\end{equation*}
\end{lema}

\begin{proof}
    Take coordinates $(s,\theta)$ on $\Sigma$, and define
    \begin{equation*}
        {g_t}_{11}=|| {\Phi _t}_*(\partial_s)||^2,\quad {g_t}_{22}=||{\Phi_t}_*(\partial_\theta)||^2,\;
        \text{and}\,\quad {g_t}_{12}=\langle {\Phi_t}_*(\partial_s),{\Phi_t}_*(\partial_\theta)\rangle. 
    \end{equation*}
    We have 
    \begin{equation*}
    E(\Phi_t)=\frac{1}{2}\int_\Sigma({g_t}_{11}+{g_t}_{22}) dsd\theta, 
    \end{equation*}
    and 
    \begin{equation*}
    A(\Phi_t)=\int_ \Sigma \sqrt{{g_t}_{11}{g_t}_{22}-({g_t}_{12})^2}  dsd\theta.
    \end{equation*}
    Suppose that at $t=0$ we have isothermal coordinates, i.e.$g_{11}=g_{22}=\rho$ and $g_{12}=0$. (We dropped the subscript $0$, so to simplify the notation). Since $\dot{\Phi}=V$ is conformal,
    \begin{equation*}
       \dot{g}_{11}= \frac{d}{dt}_{|t=0} {g_t}_{11}=\langle\nabla_V \Phi_*(\partial_s),\Phi_*(\partial_s)\rangle =\xi {g}_{11}=\dot{g}_{22},
    \end{equation*}
for some function $\xi$, and
    \begin{equation*}
         \dot{g}_{12}=\langle\nabla_V \Phi_*(\partial_s),\Phi_*(\partial_\theta)\rangle = 0. 
    \end{equation*}
Therefore, if $\ddot{g}_{11}=\frac{d^2}{dt^2}_{|t=0}{g_t}_{11}$ and $\ddot{g}_{22}=\frac{d^2}{dt^2}_{t=0}{g_t}_{22}$ we have
\begin{equation*}
   Q(V,V)= \frac{d^2}{dt^2}_{|t=0} E(\Phi_t)=\frac{1}{2}\int_\Sigma( \ddot{g}_{11}+\ddot{g}_{22})dsd\theta.
\end{equation*}
On the other hand,
\begin{align*}
    \mathcal{I}(V^\perp,V^\perp)&=\frac{d^2}{dt^2}_{|t=0} A(\Phi_t)=\frac{d^2}{dt^2}_{|t=0}\int_\Sigma \sqrt{{g_t}_{11}{g_t}_{22}-({g_t}_{12})^2}dsd\theta\\
    &=\frac{1}{2} \frac{d}{dt}_{|t=0}\int_\Sigma \frac{(\dot{g}_t)_{11}{g_t}_{22}+{g_t}_{11}(\dot{g}_t)_{22}-2{g_t}_{12}(\dot{g}_t)_{12}}{({g_t}_{11}{g_t}_{22}-({g_t}_{12})^2)^{1/2}}dsd\theta\\
    &=\frac{1}{2}\int \left(-\dfrac{1}{2}\rho^{-3}(\rho \dot{g}_{11}+\rho \dot{g}_{22})^2+2\rho^{-1}\dot{g}_{11}\dot{g}_{22}+\ddot{g}_{11}+\ddot{g}_{22} \right)dsd\theta\\
    &=\dfrac{1}{2}\int_\Sigma \left(- 2 \rho \xi^2 +2 \rho \xi^2 + \ddot{g}_{11}+\ddot{g}_{22}\right)dsd\theta \\
    &=\dfrac{1}{2}\int_\Sigma \ddot{g}_{11}+\ddot{g}_{22}dsd\theta.
\end{align*}

Therefore, for conformal vector fields along $\Sigma$, we have
\begin{equation*}
    \mathcal{I}(V^\perp,V^\perp)=Q(V,V).
\end{equation*}
\end{proof}

Now let us specialize $(\Sigma, g)$ to be a M\"obius band equipped with a Riemannian metric $g$. Then there exists a constant $s_0 > 0$ such that $(\Sigma, g)$ is conformally equivalent to the cylinder $[-s_0, s_0] \times \mathbb{S}^1$ with the identification $(s, \theta) \sim (-s, \theta + \pi)$
endowed with the flat cylindrical metric $g_0 = ds^2 + d\theta^2$. More precisely, there exists a positive smooth conformal factor $\rho$ such that $g = \rho(s, \theta) (ds^2 + d\theta^2)$,  where $\rho$ respects the symmetry of the M\"obius band, i.e. $
\rho(s, \theta) = \rho(-s, \theta + \pi).$

Let $\partial_\theta$ be the globally defined tangent vector field associated with the angular coordinate $\theta$. Observe that $\partial_\theta(p)$ spans $T_p \partial\Sigma$ at $p \in \partial\Sigma$, and the outward unit conormal of $\partial\Sigma$ in $(\Sigma,g)$ is either $\frac{1}{\sqrt{\rho}}\partial_s$ or $-\frac{1}{\sqrt{\rho}}\partial_s$. 

For the map $\Phi:(\Sigma,g)\rightarrow \mathbb{B}^{n}(r)$, define $\Phi_\theta = \Phi_*(\partial_\theta) \in \Gamma(\Phi^* T\mathbb{S}^n)$.We will prove that this vector field is in the null space of the energy index form.

\begin{lema}\label{Lemma: Null space of Q}
    Let $\Phi = (\phi_0,\dots,\phi_n) \colon (\Sigma,g) \to \mathbb{B}^n(r)$ be a free boundary minimal isometric immersion of a Möbius band. Then $\Phi_\theta$ belongs to the null space of $Q$ in the sense that
    \begin{equation*}
        Q(V, \Phi_\theta) = 0,
    \end{equation*}
    for any $V \in \Gamma(\Phi^* T\mathbb{S}^n)$ satisfying $V_p \in T_{\Phi(p)} \partial \mathbb{B}^n(r)$ for all $p \in \partial\Sigma$.
\end{lema}
\begin{proof}
By the minimality and free boundary condition, the coordinate functions satisfy
\begin{align*}
\Delta_g \varphi_i + 2\varphi_i &= 0, \quad \text{in } \Sigma, \quad i = 0, 1, \ldots, n,  \\
\frac{\partial \varphi_0}{\partial \nu} + (\tan r)\varphi_0 &= 0, \quad \text{on } \partial \Sigma,  \\
\frac{\partial \varphi_i}{\partial \nu} - (\cot r)\varphi_i &= 0, \quad \text{on } \partial \Sigma, \quad i = 1, \ldots, n. 
\end{align*}
Observe that $\Delta_g = \rho^{-1}\Delta_{g_0}=\rho^{-1}(\partial_s\partial_s+\partial_\theta\partial_\theta)$. Thus,  the first equation is equivalent to \( \Delta_{g_0}\phi_i+2\rho \phi_i=0 \), and we can differentiate it with respect to $\theta$ to obtain 
\begin{align*}
    0&=\frac{\partial}{\partial\theta}\left(\Delta_{g_0}\phi_i+2\rho\phi_i \right)\\
    &=\Delta_{g_0}\frac{\partial\phi_i}{\partial\theta}+2\frac{\partial\rho}{\partial\theta}\phi_i+2\rho \frac{\partial\phi_i}{\partial\theta}.
\end{align*}

Let  $\Phi_\theta=\left(\frac{\partial \phi_0}{\partial \theta}, \cdots, \frac{\partial \phi_n}{\partial \theta}\right)$. Then we obtained the equation 
\begin{equation}\label{eq: derivada angular fator conforma}
    \Delta_g \Phi_\theta+2\Phi_\theta+ 2\rho^{-1}\frac{\partial\rho}{\partial\theta}\Phi=0,
\end{equation}
where the Laplacian is taken coordinate to coordinate. As we have \(\langle V, \Phi\rangle=0\) for all $V\in\Gamma(\Phi^* T\mathbb{S}^n)$, we have 
\[
0=\langle V, \Delta_{g} \Phi_\theta+2\frac{\partial\rho}{\partial\theta}\Phi+2\rho \Phi_\theta\rangle=\langle V, \Delta_{g} \Phi_\theta+2\rho \Phi_\theta\rangle,
\]
for all $V\in \Gamma(\Phi^* T\mathbb{S}^n)$.

Consider \( V=(V_0,\cdots, V_n)\). Then
\begin{align*}
    0&=\int_\Sigma \langle V, \Delta _g\Phi_\theta\rangle  +2\langle V,\Phi_\theta\rangle \dvg\\
    &=\int_\sum \text{div}_g\left( \sum_{j=0}^n V_j \nabla^g \frac{\partial\phi_j}{\partial\theta} \right)-\sum_{j=0}^n\langle \nabla^g V_j, \nabla^g \frac{\partial\phi_j}{\partial\theta}\rangle+2\langle V,\Phi_\theta\rangle \dvg\\
    &=-\int_\Sigma \sum_{j=0}^n\langle \nabla^g V_j,\nabla^g \frac{\partial\phi_j}{\partial\theta}\rangle+2 \langle V,\Phi_\theta \rangle \dvg+\int_{\partial\Sigma} \sum_{j=0}^n V_j \langle \nabla^g \frac{\partial \phi_j}{\partial\theta},\nu_g \rangle dl_g\\
        &=-\int_\Sigma \sum_{j=0}^n\langle \nabla^g V_j,\nabla^g \frac{\partial\phi_j}{\partial\theta}\rangle+2 \langle V,\Phi_\theta \rangle \dvg\\
        &\quad +\int_{\partial\Sigma} \sum_{j=0}^n V_j \left[ \frac{\partial}{\partial\theta} \langle \nu_g,\nabla^g \phi_ j\rangle-\frac{\partial}{\partial\theta} \rho^{-1/2}\langle  \nu_{g_0}, \nabla^g \phi_j \rangle \right] dl_g\\
         &=-\int_\Sigma \sum_{j=0}^n\langle \nabla^g V_j,\nabla^g \frac{\partial\phi_j}{\partial\theta}\rangle+2 \langle V,\Phi_\theta \rangle \dvg\\
        &\quad +\cot{r}\int_{\partial\Sigma} \langle V,\Phi_\theta\rangle- \frac{\partial}{\partial\theta} \rho^{-1/2} \langle V, \Phi \rangle dl_g\\
        &=-\int_\Sigma \sum_{j=0}^n\langle \nabla^g V_j,\nabla^g \frac{\partial\phi_j}{\partial\theta}\rangle+2 \langle V,\Phi_\theta \rangle \dvg +\cot{r}\int_{\partial\Sigma} \langle V,\Phi_\theta\rangle dl_g.
\end{align*}

On the other hand, observe that, for any $V,W \in \Gamma(\Phi^* T \mathbb{S}^n)$, we have

\begin{equation*}
   \sum_{j=0}^n \langle \nabla^g V_j, \nabla^g W_j\rangle=\langle \overline{\nabla} V,\overline{\nabla} W\rangle+ \langle V^\top, W^\top\rangle. 
\end{equation*}
Finally we have 

\begin{align*}
    0&=\int_{\Sigma} -\langle \overline{\nabla} V, \overline{\nabla}\Phi_\theta\rangle-\langle V^\top, \Phi_\theta^\top\rangle+2\langle V, \Phi_\theta\rangle \dvg+\cot{r}\int_{\partial\Sigma} \langle V, \Phi_\theta\rangle = - Q(V,\Phi_\theta).
\end{align*}
\end{proof}

  In the next proposition, we refine \textbf{Lemma \ref{Lemma: dimesion of the V}}, by showing that the normal component of $\partial_0$ is linearly independent to the normal vector fields defined in \eqref{eq:variation_field}.
    
\begin{lema}\label{Lemma: dimesion of the sapce}
Suppose that $\Sigma$ is not contained in a plane passing through the origin. Let
\begin{equation*}
    \mathcal{C}=\text{span}\{V_1,\cdots, V_n, \partial_0^\perp\},
\end{equation*}
where $V_i=V_{e_i}$ as in \eqref{eq:variation_field}, for every $i=1,\ldots,n$. Then $\mathcal{C}$ is $(n+1)$-dimensional vector space of normal vector fields long $\Sigma$.
\end{lema}
\begin{proof}
Let \( a_0, \dots, a_n \in \mathbb{R} \) be such that
    \[
    a_0 \partial_0^\perp + \sum_{i=1}^n a_i V_i = 0.
    \]
We want to show that $a_i=0$ for every $i=0,\ldots,k$. Consider the vector \( y \in \mathbb{R}^{n+1} \) defined by
    \[
    y = \sum_{i=1}^n a_i e_i.
    \]
Observe that \( y \perp e_0 \). Then we have
    \[
    \left( a_0 + \langle p, y \rangle \right) \partial_0^\perp + \left( \frac{1 + \sin r}{\cos r} - \phi_0 \right) \partial_y^\perp = 0.
    \]
    
If \( y \neq 0 \), observe that for \( p \in \partial \Sigma \), we have \( \partial_0^\perp(p) = 0 \), and hence \( \partial_y^\perp|_{\partial \Sigma} = 0 \). As in the proof of \textbf{Lemma \ref{Lemma: dimesion of the V}}, we conclude that $\Sigma$ is orthogonal to the fixed vector $y$, which contradicts our assumptions. 

Hence, \( y = 0 \), meaning that $a_1=\ldots=a_n=0$. If $a_0\neq 0$, then \( \partial_0^\perp(p) = 0 \) for all \( p \in \Sigma \). In this case $\Sigma$ is contained in the union of orbits of the flow of the conformal vector field $\partial_0$, which are contained in great circles of $\mathbb{S}^n$. It follows that \( \Sigma \) is totally geodesic in $\mathbb{S}^n$, which again contradicts the assumption that $\Sigma$ is not contained in a plane through the origin.

Thus, the only possibility is that $a_i=0$ for every $i=0,\ldots,n$, as we wanted to prove.
\end{proof}

We remark that, as the above proof shows, Lemma \ref{Lemma: dimesion of the sapce} holds in arbitrary dimensions and codimensions, and can be seen as an generalization of Lemma \ref{Lemma: dimesion of the V}.\\

We now proceed to prove that there exists a subspace of dimension \( n \) in \( \mathcal{C} \), such that each element of this subspace can be deformed by a tangential vector field in a way that the resulting deformation vector field is conformal.

\begin{prop}\label{Prop:ConformalIndex}
    Let \( \Phi: (\Sigma, g) \to \mathbb{B}^n(r) \) be a free boundary minimal immersion of a Möbius band. Then there exists a subspace \( \mathcal{C}_1 \subset \mathcal{C} \) of dimension at least \( n \) such that, for every \( V \in \mathcal{C}_1 \), the following properties hold:
    \begin{enumerate}
        \item There exists a tangential vector field \( Y^\top \in \Gamma(T\Sigma) \), tangent to \( \partial \mathbb{B}^n(r) \) along \( \partial \Sigma \);
        
        \item The deformation field \( Y(V) = V + Y^\top \) is conformal;
        
        \item If \( Y_1^\top \) and \( Y_2^\top \) are tangential vector fields such that \( V + Y_1^\top = V + Y_2^\top \) is conformal, then there exists a constant \( c \in \mathbb{R} \) such that
        \[
        \Phi_*(Y_1^\top - Y_2^\top) = c\, \Phi_\theta;
        \]
        
        \item The associated quadratic forms satisfy
        \[
        Q(Y(V), Y(V)) = \mathcal{I}(V, V).
        \]
    \end{enumerate}
\end{prop}
\begin{proof}
Let \( \tilde{\Sigma} = [-T, T] \times \mathbb{S}^1 \) be the orientable double cover of \( \Sigma \). Let $\tau$ be the deck trasformation i.e, $\tau(s,\theta)=(-s,\theta+\pi)$

Our first goal is to find \( Y^\top \in \Gamma(T\tilde{\Sigma}) \) such that \( Y = Y(V) = V + Y^\top \), for $V\in \mathcal{C}$, is conformal and satisfies the symmetry condition
\[
Y ^\top= Y^\top \circ \tau.
\]
The conformality condition yields the following equations:
\begin{equation*}
\langle \nabla_{\partial_\theta} Y, \partial_s \rangle + \langle \nabla_{\partial_s} Y, \partial_\theta \rangle = 0, \quad \text{and} \quad \langle \nabla_{\partial_s} Y, \partial_s \rangle = \langle \nabla_{\partial_\theta} Y, \partial_\theta \rangle.
\end{equation*}

Expanding the first condition, and using that $V\in \mathcal{C}$ is normal:
\begin{align*}
0 &= \langle \nabla_{\partial_\theta}(Y^\top + V), \partial_s \rangle + \langle \nabla_{\partial_s}(Y^\top + V), \partial_\theta \rangle \\
&= \langle \nabla_{\partial_\theta} Y^\top, \partial_s \rangle + \langle \nabla_{\partial_s} Y^\top, \partial_\theta \rangle - 2 \langle V, B(\partial_s, \partial_\theta) \rangle.
\end{align*}
Similarly,the second condition becomes:
\begin{align*}
0 &= \langle \nabla_{\partial_\theta}(Y^\top + V), \partial_\theta \rangle - \langle \nabla_{\partial_s}(Y^\top + V), \partial_s \rangle \\
&= \langle \nabla_{\partial_\theta} Y^\top, \partial_\theta \rangle - \langle \nabla_{\partial_s} Y^\top, \partial_s \rangle \\
&\quad - \langle V, B(\partial_\theta, \partial_\theta) \rangle + \langle V, B(\partial_s, \partial_s) \rangle \\
&= \langle \nabla_{\partial_\theta} Y^\top, \partial_\theta \rangle - \langle \nabla_{\partial_s} Y^\top, \partial_s \rangle + 2 \langle V, B(\partial_s, \partial_s) \rangle,
\end{align*}
where in the last line we use the fact that \( \Sigma \) is minimal.

Write \( Y^\top = u \Phi_s + v \Phi_\theta \), and note that the metric is given by \( g = \rho (ds^2 + d\theta^2) \). Then,
\[
\nabla_X Y = \nabla^{g_0}_X Y + \frac{1}{2\rho} \left[ X(\rho) Y + Y(\rho) X - g_0(X, Y) \nabla^{g_0} \rho \right],
\]
and since \( \{ \partial_s, \partial_\theta \} \) is a parallel frame for \( g_0 \), we compute:

\begin{align*}
\langle \nabla_{\partial_s} Y^\top, \partial_s \rangle 
&= \langle \nabla_{\partial_s}^{g_0}(u \partial_s + v \partial_\theta), \partial_s \rangle \\
&\quad + \frac{1}{2\rho} \langle (\partial_s \rho)(u \partial_s + v \partial_\theta) + (u \partial_s \rho + v \partial_\theta \rho) \partial_s \\
& \quad- u[(\partial_s \rho) \partial_s + (\partial_\theta \rho) \partial_\theta], \partial_s \rangle \\
&= u_s \rho + \frac{1}{2} (u \rho_s + v \rho_\theta).
\end{align*}

Similarly, we find:
\[
\langle \nabla_{\partial_\theta} Y^\top, \partial_\theta \rangle = v_\theta \rho + \frac{1}{2}(v \rho_\theta + u \rho_s),
\]
\[
\langle \nabla_{\partial_\theta} Y^\top, \partial_s \rangle = u_\theta \rho + \frac{1}{2}(u \rho_\theta - v \rho_s),
\]
\[
\langle \nabla_{\partial_s} Y^\top, \partial_\theta \rangle = v_s \rho + \frac{1}{2}(v \rho_s - u \rho_\theta).
\]

Putting all this together, we obtain the system:
\begin{equation*}
\begin{cases}
u_\theta + v_s = \dfrac{2}{\rho} \langle V, B(\partial_\theta, \partial_s) \rangle, \\ \hfill \\
u_s - v_\theta = \dfrac{2}{\rho} \langle V, B(\partial_s, \partial_s) \rangle.
\end{cases}
\end{equation*}

Define \( f = u + i v \) and \( z = s + i\theta \), then the above system becomes a first-order complex PDE:
\[
\frac{\partial f}{\partial \overline{z}} = h,
\]
where 
\begin{equation}\label{eq: definicao de h}
h = \frac{2}{\rho} \langle V, B(\partial_s, \partial_s) \rangle + i \frac{2}{\rho} \langle V, B(\partial_\theta, \partial_s) \rangle.
\end{equation}

Notice that $h$ satisfies
\begin{equation}\label{eq:Invariancia_h}
    h\circ\tau=\overline{h}
\end{equation}

In fact, as we have  $V=V\circ\tau$,  $\partial_s=-\partial_s\circ\tau$ and $\partial_\theta=\partial_\theta \circ \tau$ we can compute
\begin{align*}
    h(-s,\theta+\pi)&= \frac{2}{\rho\circ\tau} \langle V\circ\tau, B(\partial_s, \partial_s)\circ\tau \rangle \\
    &-+i \frac{2}{\rho\circ\tau} \langle V\circ\tau, B(\partial_\theta, \partial_s)\circ\tau \rangle.\\
    &= \frac{2}{\rho} \langle V, B(-\partial_s,- \partial_s) \rangle
    + i \frac{2}{\rho} \langle V, B(\partial_\theta, -\partial_s) \rangle=\overline{h}(s,\theta).
\end{align*}
In order to ensure that \( Y^\top \) is tangent to \( \partial \mathbb{B}^n(r) \) along \( \partial \Sigma \), we impose the boundary condition \( u = 0 \) on \( \partial \Sigma \), i.e., \( \Re(f) = 0 \) on \( \partial \Sigma \), where \(\Re\) denotes the real part.

Thus, we are led to the following boundary value problem:
\begin{equation}\label{eq: system}
\begin{cases}
\dfrac{\partial f}{\partial \overline{z}} = h & \text{in } \tilde{\Sigma}, \\
\Re(f) = 0 & \text{on } \partial \tilde{\Sigma}.
\end{cases}
\end{equation}

Before we analyze the existence of solutions to this problem, let us observe that any two solutions $f_1$ and $f_2$ to \eqref{eq: system} differ by a imaginary constant. In fact, the real part of  $f_1-f_2$ vanish at the boundary and by the maximum principle vanish at the interior. This remark has also the consequence that any solution $f$ on $\tilde{\Sigma}$ to \eqref{eq: system} satisfy $f\circ \tau=-\overline{f}$. In fact, since   $h$ satisfies \eqref{eq:Invariancia_h} and observe that

\[\frac{\partial }{\partial z} (f\circ\tau)= \left (\frac{\partial f }{\partial z}\circ\tau\right)\frac{\partial \tau}{\partial z}+\left (\frac{\partial f }{\partial \overline{z}}\circ\tau\right)\frac{\partial \overline{\tau}}{\partial z}=-1 \frac{\partial f }{\partial \overline{z}}\circ\tau =-\overline{h},\]

Therefore we have \(-h= \overline{\frac{\partial }{\partial z} (f\circ\tau)}=\frac{\partial }{\partial \overline{z}} (\overline{f\circ\tau})\). Then, both $f$ and $-\overline{f\circ \tau}=-\overline{f}\circ \tau$ are solutions to \eqref{eq: system}, and therefore satisfy $f+\overline{f}\circ \tau = ic$ for some constant $c\in \mathbb{R}$. Since $\tau^2=Id_{\tilde{\Sigma}}$, we conclude that $c$ must vanish, as claimed.

Let us consider the function space
\[
\mathcal{F}_1 = \left\{ f \in H^1(\tilde{\Sigma}, \mathbb{C}) ; \Re(f) = 0 \text{ on } \partial \tilde{\Sigma} \right\}.
\]

The system \eqref{eq: system} is an elliptic boundary value problem on $\tilde{\Sigma}$. By the Fredholm alternative, the problem \eqref{eq: system} has a solution in \( \mathcal{F}_1 \) if and only if the inhomogeneous term \( h \) is orthogonal (in \( L^2(\tilde{\Sigma}, g_0) \)) to the kernel of the adjoint operator. 

The adjoint of \( \dfrac{\partial}{\partial \overline{z}} \) is \( -\dfrac{\partial}{\partial z} \), defined on the domain:
\[
\mathcal{F}_2 = \left\{ \xi \in H^1(\tilde{\Sigma}, \mathbb{C}) ; \mathfrak{Im}(\xi) = 0 \text{ on } \partial \tilde{\Sigma} \right\}.
\]

Then \( \xi \) is in the kernel of the adjoint if, and only if,:
\begin{equation*}\label{eq: adjointhomogenous}
\begin{cases}
\dfrac{\partial \xi}{\partial z} = 0 & \text{in } \tilde{\Sigma}, \\
\mathfrak{Im}(\xi) = 0 & \text{on } \partial \tilde{\Sigma}.
\end{cases}
\end{equation*}

By the maximum principle applied to the imaginary part of a solution $\xi$ of \ref{eq: adjointhomogenous}, we conclude that its problem are real constant functions, i.e. \( \xi \equiv c \in \mathbb{R} \). Hence, a solution of \eqref{eq: system} exists in $\mathcal{F}_1$ if, and only if, the function $h$ as in \eqref{eq: definicao de h} is such that its real part is orthogonal in to constant functions.

Define:
\[
\mathcal{C}_1 = \left\{ V \in \mathcal{C} : \Re \left( \int_{\tilde{\Sigma}} \left[\frac{2}{\rho} \langle V, B(\partial_s, \partial_s) \rangle +i \frac{2}{\rho} \langle V, B(\partial_\theta, \partial_s) \rangle \right] \right) = 0 \right\}.
\]

This is a codimension-one condition on the $n+1$-dimensional space $\mathcal{C}$ (see \textbf{Lemma} \ref{Lemma: dimesion of the sapce}), so \( \dim \mathcal{C}_1 \geq n \). As we have proved, for every $V\in \mathcal{C}_1$ there exists a vector $Y=Y(V)$ satisfying items $(1)$ and $(2)$ of the proposition.

Regarding item $(3)$, let \( Y_k^\top = u_k \Phi_s + v_k \Phi_\theta \) as in $(1)$ and $(2)$ for the same $V\in \mathcal{C}_1$, and define \( f_k = u_k + i v_k \), for \( k = 1, 2 \). As we have shown above, both $f_1$ and $f_2$ satisfy the same equation \eqref{eq: system}, and therefore
 \( f_1 - f_2 = i c \), with \( c \in \mathbb{R} \). Hence,
\[
\Phi_*(Y_1^\top - Y_2^\top) = (u_1 - u_2) \Phi_s + (v_1 - v_2) \Phi_\theta = c \Phi_\theta.
\]

The final statement (item $(4)$) follows immediately from \textbf{Lemma} \ref{lemma: Relacao da variacao de energia e da area}.
\end{proof}

\begin{thmA}\label{Thm: Uniqueness_of_Mobius_band}
        Let \((\Sigma, g)\) be a Möbius band, and let \(\Phi = (\phi_0, \ldots, \phi_n): (\Sigma, g) \to \mathbb{B}^n(r)\) be a free boundary minimal immersion by frist Steklov eigenfunction. Then \(\Sigma\) is intrinsically rotationally symmetric. 
\end{thmA}

\begin{proof}
We distinguish two cases:

\textbf{Case 1:} Suppose that
\[
\int_{\partial\Sigma} \frac{\partial \Phi}{\partial\theta} \,\dvl = 0.
\]
Then we may use \( \frac{\partial \phi_i}{\partial\theta} \) as a test function for \( \sigma_1 = \cot r \), for \( i = 1, \dots, n \). This gives
\begin{equation} \label{eq:testfunction_characterization_rotational}
\int_{\Sigma} \left|\nabla^g \frac{\partial \phi_i}{\partial\theta} \right|^2 - 2\left( \frac{\partial \phi_i}{\partial\theta} \right)^2 \,\dvg - \sigma_1 \int_{\partial \Sigma} \left( \frac{\partial \phi_i}{\partial\theta} \right)^2 \, \dvl \ge  0.
\end{equation}

Also, observe that \( \frac{\partial \phi_0}{\partial \theta} = 0 \) on \( \partial \Sigma \). Then, by Proposition 2 in \cite{LM23}, we have
\[
\int_{\Sigma} \left| \nabla^g \frac{\partial \phi_0}{\partial\theta} \right|^2 - 2\left( \frac{\partial \phi_0}{\partial\theta} \right)^2 \, \dvg \ge 0.
\]

On the other hand, \textbf{Lemma} \ref{Lemma: Null space of Q} gives \( Q(\Phi_\theta, \Phi_\theta) = 0 \). Computing the second variation of energy, we obtain:
\[
0 = Q(\Phi_\theta, \Phi_\theta) = \sum_{i=1}^n \left( \int_{\Sigma} \left| \nabla^g \frac{\partial \phi_i}{\partial\theta} \right|^2 - 2\left( \frac{\partial \phi_i}{\partial\theta} \right)^2 \, \dvg- \int_{\partial \Sigma} \frac{\partial \phi_i}{\partial\theta} \left\langle \frac{\partial \phi_i}{\partial\theta}, \nu \right\rangle \, \dvl\right).
\]

Now, recall that

\[
\left\langle \frac{\partial \phi_i}{\partial\theta}, \nu \right\rangle = \sigma_1 \frac{\partial \phi_i}{\partial\theta} + \rho^{-1/2} \frac{\partial \rho}{\partial\theta} \sigma_1 \phi_i.
\]

Since \( \left. \frac{\partial \phi_0}{\partial \theta} \right|_{\partial \Sigma} = 0 \) and \( \langle \Phi, \Phi_\theta \rangle = 0 \), we obtain
\[
\sum_{i=0}^n \int_{\partial \Sigma} \frac{\partial \phi_i}{\partial\theta} \left\langle \frac{\partial \phi_i}{\partial\theta}, \nu \right\rangle \, \dvl = -\sigma_1 \sum_{i=1}^n \int_{\partial \Sigma} \left( \frac{\partial \phi_i}{\partial\theta} \right)^2 \, \dvl.
\]

Hence, equality holds in \eqref{eq:testfunction_characterization_rotational}. Furthermore, equation \eqref{eq: derivada angular fator conforma} gives:
\[
\Delta_g \left( \frac{\partial \phi_i}{\partial \theta} \right) + 2\rho^{-1} \frac{\partial \rho}{\partial \theta} \phi_i + 2\frac{\partial \phi_i}{\partial \theta} = 0, \quad \text{for } i = 0, \dots, n.
\]
We conclude that \( \rho^{-1} \frac{\partial \rho}{\partial \theta} \phi_i = 0 \) on $ \Sigma$ for all \( i = 1, \dots, n \), hence \( \frac{\partial \rho}{\partial \theta} = 0 \).

\textbf{Case 2:} Suppose now that the vector
\[
\int_{\partial \Sigma} \frac{\partial \Phi}{\partial \theta}\d l_g,
\]
is arbitrary. By \textbf{Proposition \ref{Prop:ConformalIndex}}, for any \( V \in \mathcal{C}_1 \), there exists a conformal vector field \( Y(V) = V + Y^\top \), whose normal component is \( V \). Moreover, the field \( Y(V) \) is uniquely determined if we impose the orthogonality condition
\[
\left\langle \int_{\partial \Sigma} Y(V), \int_{\partial \Sigma} \Phi_\theta \right\rangle = 0,
\]
which follows from item (3) of \textbf{Proposition \ref{Prop:ConformalIndex}}. The map \( V \mapsto Y(V) \in \Gamma(\Phi^* T\mathbb{S}^n) \) is well-defined and linear.

In addition, since \( Y^\top \) is tangent to \( \partial \mathbb{B}^n(r) \) along \( \partial \Sigma \), it follows from the free boundary condition that \( Y(V)(p) \in T_p \partial \mathbb{B}^n(r) \) for every \( p \in \partial \Sigma \). Thus, on these points,
\[
0 = \langle Y(V), \nu \rangle = \left\langle Y(V), \frac{1}{\sin r}(\cos r \, p - e_0) \right\rangle \quad \Rightarrow \quad \langle Y(V), e_0 \rangle = 0
\]
Define the vector space
\[
\mathcal{W} = \{ Y(V) + c\Phi_\theta \mid V \in \mathcal{C}_1,\; c \in \mathbb{R} \}.
\]
Note that \( \dim \mathcal{W} \geq  n + 1 \). Furthermore, for all \( W \in \mathcal{W} \), the first coordinate of \( W \) vanishes along \( \partial \Sigma \), i.e.
\[
W_0 = \langle Y(V) + c\Phi_\theta, e_0 \rangle = 0.
\]

Define the linear map \( T: \mathcal{W} \to \mathbb{R}^{n+1} \) by
\[
T(W) = \int_{\partial \Sigma} W = \left( \int_{\partial \Sigma} W_0, \dots, \int_{\partial \Sigma} W_n \right).
\]
Since \( W_0 = 0 \), it follows that \( \operatorname{Im}(T) \subseteq \{0\} \times \mathbb{R}^n \). By the rank–nullity theorem, there exists \( V \in \mathcal{C}_1 \), non-both zero, and \( c \in \mathbb{R} \) such that
\begin{equation} \label{eq:testfunction_uniqueness}
T\left( Y(V) + c\Phi_\theta \right) = \int_{\partial \Sigma} Y(V) + c\Phi_\theta \, \dvl = 0.
\end{equation}

By \textbf{Lemmas} \ref{lemma: Relacao da variacao de energia e da area} and \ref{Lemma: Null space of Q}, we have
\[
Q(Y(V) + c\Phi_\theta, Y(V) + c\Phi_\theta) = Q(Y(V), Y(V)) = \mathcal{I}(V, V).
\]

Now let us write \( V = a V_y + b \partial_0^\perp \), with \( a, b \in \mathbb{R} \), and \( y \perp e_0 \). Then, by Theorem \ref{thm:index_estimate} and the equation \eqref{eq:Jacobi_operator_on_V_y}  done in Section \ref{sec: index estimate}, and the fact that \(\partial_0^\perp|_{\partial\Sigma}=0\) and $V_y|_{\partial\Sigma}$ is tangent to \(\partial\Br \) Then the boundary terms in \( \mathcal{I}(V_y,\partial_0^\perp)\) and  \( \mathcal{I}(\partial_0^\perp,\partial_0^\perp)\) vanish, we get 
\begin{align*}
\mathcal{I}(V, V) &= a^2 \mathcal{I}(V_y^\perp, V_y^\perp) + 2ab \mathcal{I}(V_y^\perp, \partial_0^\perp) + b^2 \mathcal{I}(\partial_0^\perp, \partial_0^\perp) \\
&= a^2 \int_{\Sigma} -k c^2(r) |\partial_y^\perp|^2\dvg + 2ab \int_{\Sigma} \langle -L_\Sigma V_y^\perp, \partial_0^\perp \rangle\dvg \\
& \quad + b^2 \int_{\Sigma} \langle -L_\Sigma \partial_0^\perp, \partial_0^\perp \rangle\dvg \\
&=\int_\Sigma -k a^2 c^2(r) \|\partial_y\|^2-k2abc(r)\langle\partial_y^\perp,\partial_0^\perp\rangle-kb^2\|\partial_0^\perp\|^2\dvg \\
&= -k \int_{\Sigma} \left\| a c(r)\partial_y^\perp + b \partial_0^\perp \right\|^2 \dvg\le 0.
\end{align*}
Thus, \( Q(Y(V) + c\Phi_\theta, Y(V) + c\Phi_\theta) \le 0 \). 

Define \( f_i = \langle Y(V) + c\Phi_\theta, e_i \rangle \), for \( i = 0, \dots, n \). By equation \eqref{eq:testfunction_uniqueness}, the components \( f_i \), for \( i = 1, \dots, n \), are test functions for \( \sigma_1 = \cot r \). Hence,
\[
\int_{\Sigma} |\nabla^g f_i|^2 - 2 f_i^2 \, \dvg + \cot r \int_{\partial \Sigma} f_i^2 \, \dvl \ge 0.
\]

Moreover, since \( f_0(p) = 0 \) for all \( p \in \partial \Sigma \), Proposition 2 of \cite{LM23} implies:
\[
\int_{\Sigma} |\nabla^g f_0|^2 - 2 f_0^2 \, \dvg \ge 0.
\]

Therefore,
\begin{align*}
Q(Y(V) + c\Phi_\theta, Y(V) + c\Phi_\theta) &= \sum_{i = 0}^n \left( \int_{\Sigma} |\nabla^g f_i|^2 - 2 f_i^2 \, 
\dvg+ \cot r \int_{\partial \Sigma} f_i^2 \, \dvl \right) \\
&\ge 0.
\end{align*}

But we also had \( I(V,V)= Q(Y(V) + c\Phi_\theta, Y(V) + c\Phi_\theta) \le 0 \), so equality must hold. By \textbf{Theorem~\ref{thm:index_estimate}}, $V=0$. In particular, by \eqref{eq:testfunction_uniqueness},
\[
\int_{\partial \Sigma} \Phi_\theta \dvl = 0,
\]
which reduces us to the situation already treated in Case 1.
\end{proof}

\section{Free boundary minimal Möbius band in spherical caps} \label{sec:ODE}

For $0<r\le \frac{\pi}{2}$, consider the spherical cap $\mathbb{B}^n(r)\subset \mathbb{S}^n$. We aim to construct an example of a free boundary minimal Möbius band in $\mathbb{B}^n(r)$. Free boundary minimal surfaces are closely related to spectral problems. Specifically, as shown in \textbf{Theorem A}, if this surface is immersed by its first Steklov eigenfunction, then it must be intrinsically rotationally symmetric. This symmetry reduces the problem of finding such surfaces to the problem of solving a system of ordinary differential equations satisfying specific boundary conditions.

\subsection{Rotational symmetry} Let $s_0 > 0$. Consider the annulus  $[-s_0, s_0] \times \mathbb{S}^1 $  equipped with the flat metric $ g_0 = ds^2 + d\theta^2 $, and define the equivalence relation $ (s, \theta) \sim (-s, \theta + \pi)$. The quotient space $[-s_0, s_0] \times \mathbb{S}^1/\sim$ is diffeomorphic to the Möbius band, and the metric $g_0$ passes down to it as a flat metric. 

Any Riemannian Möbius band $(\mathbb{M},g)$ is conformally equivalent to the quotient space $([-s_0, s_0] \times \mathbb{S}^1 / \sim ,g_0)$, for some $s_0>0$. That is, there exist a modulus $s_0> 0$ and a positive function $ \rho \colon [-s_0, s_0] \times \mathbb{S}^1 \to \mathbb{R}$ satisfying the symmetry condition $\rho (s, \theta) = \rho(-s, \theta + \pi)$, such that the metric $g$ can expressed as $g = \rho g_0$. We say \( g \) is intrinsically rotationally symmetric if \( \rho \) depends only on \( s \), i.e. $\frac{\partial \rho}{\partial \theta} = 0 \; \text{(equivalently, } \rho (s, \theta) = \rho (s) \text
{)}$.

Let \( \Phi: (\mathbb{M}, g) \to \mathbb{B}^r(r)\subset \mathbb{S}^n \) be an immersion, where \( \Phi = (\phi_0, \dots, \phi_n) \) and \( g = \Phi^* \text{can}_{\mathbb{S}^n} \). Recall that
 the surface \( \Phi(\mathbb{M}) \) is minimal if and only  the coordinate function satisfies the equation \eqref{eq:minimal_condition}  and is   free boundary if and only if satisfies the \eqref{eq:free_boundary_condition}  on the boundary. 

Assuming \( \mathbb{M} \) is intrinsically  rotationally symmetric, we can use separation of variables in order to solve this boundary value problem. In fact, under this symmetry, the partial differential equations \eqref{eq:minimal_condition} and the boundary conditions \eqref{eq:free_boundary_condition} reduce to a system of ODEs  with appropriate boundary conditions.

\begin{prop}\label{prop:Characterization_of_FB_minimal_mobius_band}
    Let  \( 0 < r \leq \frac{\pi}{2} \), \( a\in (0,1) \) and \( s_r > 0 \). Consider functions \( y \colon [-s_r, s_r] \to \mathbb{R} \) and \( z \colon [-s_r, s_r] \to \mathbb{R} \) satisfying the following system of ordinary differential equations:
\begin{equation}\label{eq:System_of_ODEs_FBMMB}
  \begin{cases}
      y''(s) = \big(1 - 2y^2(s) - 8z^2(s)\big) y(s), & y(0) = 0, \ y'(0) = 2a, \\
      z''(s) = \big(4 - 2y^2(s) - 8z^2(s)\big) z(s), & z(0) = a, \ z'(0) = 0,
  \end{cases}
\end{equation}
with the following additional properties:
\begin{itemize}
    \item[$(i)$] The function \( y(s) \) is odd and vanishes only at \( s = 0 \) on the interval \([-s_r,s_r]\) , while \( z(s) \) is even on the interval and never vanish.
    \item[$(ii)$] \( y^2(s) + z^2(s) \leq \sin^2 r \) for all \( s \in [-s_r, s_r] \);
    \item[$(iii)$] Defining \( x(s) = \sqrt{1 - y^2(s) - z^2(s)} \), the following boundary conditions hold:
    \begin{itemize}
        \item[(a)] \( x(-s_r) = x(s_r) = \cos r \),
        \item[(b)] \( x'(\pm s_r) = - \sin r \sqrt{y^2(\pm s_r) + 4z^2(\pm s_r)}  \).
    \end{itemize}
\end{itemize}

Then, the map \( \Phi_{(a,r)} \colon [-s_r, s_r] \times \mathbb{S}^1/\sim \, \to \mathbb{B}^4(r) \) defined by
\begin{equation}\label{eq:Phi_definition}
        \Phi_{(a,r)}(s, \theta) = \big( 
        x(s), \ y(s)\cos\theta, \ y(s)\sin\theta, \ z(s)\cos 2\theta, \ z(s)\sin 2\theta 
    \big),
\end{equation}
is intrisically rotationally symmetric free boundary minimal immersion of a  Möbius band, where $x(s)$ is a $\sigma_0$-Steklov eigenfunction and the other coordinate functions $y(s)\cos\theta, y(s)\sin\theta, z(s)\cos 2\theta, z(s)\sin 2\theta$ are $\sigma_1$-Steklov eigenfunction.
\end{prop}

\begin{proof}
Observe that $x$ is an even function and that, by definition, $\Phi(-s,\theta+\pi)=\Phi(s,\theta)$, so that $\Phi$ descends to the Möbius band $\mathbb{M}=[-s_0,s_0]\times \mathbb{S}^1/\sim$. Moreover, also by construction,
\[
\|\Phi_{(a,R)}(s,\theta)\|^2 = x^2 + y^2(\sin^2{\theta} + \cos^2{\theta}) + z^2(\sin^2{2\theta} + \cos^2{2\theta}) = 1.
\]
Therefore, the image of \(\Phi\) lies on the unit sphere. Notice that condition $(ii)$ implies that the image of \(\Phi\) is contained in the spherical cap \(B^n(r)\). 

Let \(g = \Phi^* \canSn\) be the pullback of the canonical metric on the sphere. We compute:
\[
g_{12} = g_{21} = \left\langle \frac{\partial \Phi}{\partial s}, \frac{\partial \Phi}{\partial \theta} \right\rangle = 0,
\]
\[
g_{11} = \left\langle \frac{\partial \Phi}{\partial s}, \frac{\partial \Phi}{\partial s} \right\rangle = (x')^2 + (y')^2 + (z')^2,
\]
\[
g_{22} = \left\langle \frac{\partial \Phi}{\partial \theta}, \frac{\partial \Phi}{\partial \theta} \right\rangle = y^2 + 4z^2.
\]
In the next section, we prove that \((x')^2 + (y')^2 + (z')^2 = y^2 + 4z^2\)  is a positive function, see equation~\eqref{eq: y^2+4z^2=x'^2+y'^2+z'^2}. Therefore, \(g\) is conformally equivalent to the flat rotationally symmetric metric of the form
\[
g = \rho(s)(ds^2 + d\theta^2),
\]
for $\rho(s)=y^2+4z^2$. Consequently, the Laplace--Beltrami operator associated with \(g\) is given by
\[
\Delta_g = \frac{1}{\rho(s)}\left( \frac{\partial^2}{\partial s^2} + \frac{\partial^2}{\partial \theta^2} \right).
\]

Observer that $x(s)=\sqrt{1-y^2(s)-z^2(s)}$ also satisfies a second-orderm ODE. In fact deriving twice  $x^2+y^2+z^2=1$ we obtain
\begin{align*}
   x''(s)x(s)&=-(x'(s)^2+y'(s)^2+z'(s)^2)-y''(s)y(s)-z''(s)z(s),\\
    &=-\rho(s)-(1-2\rho(s))y^2(s)-(4-2\rho(s))z^2(s),\\
    &=-y^2(s)-4z^2(s)+\rho(s)(-1+2y^2(s)+2z^2(s)),\\
    &=-2\rho(s)(-1+y^2+z^2),\\
    &=-2\rho(s)x^2(s),
\end{align*}
Then
\begin{equation}\label{eq:ode_for_x}
     x''(s)=-2\rho(s)x(s)    
\end{equation}

The sistem of  ODE  \eqref{eq:System_of_ODEs_FBMMB} and \eqref{eq:ode_for_x}  implies that  the component functions of \(\Phi\) satisfies the Laplace equation with frequency $2$ on $\Sigma$. In fact,
\begin{align*}
    (\Delta_g + 2)(y(s)\cos{\theta}) 
    &= \frac{1}{\rho(s)} \left( \frac{\partial^2}{\partial s^2} \left( y(s)\cos{\theta} \right) + \frac{\partial^2}{\partial \theta^2} \left( y(s)\cos{\theta} \right) \right) + 2y(s)\cos{\theta} \\
    &= \frac{\cos{\theta}}{\rho(s)}\left( y'' - y + 2\rho y \right) \\
    &= \frac{\cos{\theta}}{\rho(s)}\left( y'' - (1 - 2y^2 - 8z^2)y \right) \\
    &= 0.
\end{align*}

Similarly, the same holds for each of the functions $x(s),y(s)\sin{\theta}, z(s)\cos{2\theta},$ and $z(s)\sin{2\theta}$. Then $\Phi:\mathbb{M}\to \mathbb{B}^4(r)$ is minimal.

We now verify that \(\Phi\) defines a free boundary immersion. This occurs when the outward unit normal to the boundary of the ball, given by
\[
N_{\partial \mathbb{B}^n(R)} = \frac{1}{\sin r}(\cos r \, \Phi - e_0),
\]
coincides with the outward unit conormal to the boundary of the immersed surface \(\Phi(\partial \mathbb{M})\), which is given by
\[
\Phi_*(\nu_{\partial \mathbb{M}}) = \frac{1}{\sqrt{\rho(s)}} \frac{\partial \Phi}{\partial s}.
\]
Therefore, the free boundary condition is satisfied if and only if
\begin{align*}
    0 &= \left\| N_{\partial \mathbb{B}^n(R)} -\Phi_*(\nu_{\partial \mathbb{M}}) \right\|^2 \\
      &= 2 - 2 \left\langle \frac{1}{\sin r}(\cos r \, \Phi - e_0), \frac{1}{\sqrt{\rho(s)}} \frac{\partial \Phi}{\partial s} \right\rangle \\
      &= 2 + 2 \frac{x'(s_r)}{\sin r \sqrt{y^2(s_r) + 4z^2(s_r)}}.
\end{align*}

Thus, condition $(iii)$ guarantees that \(\Phi\) is a free boundary minimal immersion. 

Since \(x(s) > 0\), it is straightforward to verify that \(x(s)\) is a Steklov eigenfunction associated to the eigenvalue \(\sigma_0 = -\tan(r)\).
On the other hand, we checked above that
\[
\{ y(s)\cos\theta, \ y(s)\sin\theta, \ z(s)\cos 2\theta, \ z(s)\sin 2\theta \}
\]
are Steklov eigenfunctions with eigenvalue \(\sigma = \cot r\), by \eqref{eq:free_boundary_condition}. 

In order to finish the proof of the Proposition, it remains to show that \(\cot  r = \sigma_1\), where \(\sigma_1\) denotes the second Steklov eigenvalue.

It is known, by Courant nodal set theorem, that eigenfunctions of $(\mathbb{M},\rho g_0)$  corresponding to \(\sigma_1\) must have exactly two nodal domains . Moreover, by the separation of variables argument, the corresponding eigenspace is  contained in the space generated by the set
\[
\{ \phi_0(s),\ \phi_1(s)\cos\theta,\ \phi_1(s)\sin\theta,\ \phi_2(s)\cos 2\theta,\ \phi_2(s)\sin 2\theta \},
\]
where $\phi_k$ must satisfy the equation,
\begin{equation}\label{eq:equation_for_phi_k}
    \phi_k''(s)=(k^2-2\rho(s))\phi_k(s),
\end{equation}
Moreover \(\phi_0\) has two zeros in \([-s_r, s_r]\), \(\phi_1\) has exactly one zero in \([-s_r, s_r]\), and \(\phi_2\) not vanish in the interval.

Although this eigenspace is a priori spanned by five eigenfunctions, we will demonstrate that it is, in fact, generated by only four. To establish this, we prove the nonexistence of a $\sigma_1$-Steklov eigenfunction  on the Möbius strip that is independent of the  variable $\theta$. 

Observe that, by the invariance under the deck transformation, we have \(\phi_0(s) = \phi_0(-s)\) for every $s\in [-s_0,s_0]$, so that \(\phi_0\) is even and consequently satisfies \(\phi_0'(0) = 0\). Also, we know that \(x(s)\) is a positive solution \eqref{eq:equation_for_phi_k}, for $k=0$. By Abel's formula, a second linearly independent solution of this equation is given by
\[
\phi(s) = x(s)\int_{0}^s \frac{1}{x^2(t)}\,dt.
\]
Hence, any solution can be written in the form
\[
\phi_0(s) = a\, x(s) + b\, x(s)\int_{0}^s \frac{1}{x^2(t)}\,dt.
\]
where $a,b\in\R$. We compute
\[
\phi_0'(0) = a\, x'(0) + b \left( x'(0) \int_{0}^{0} \frac{1}{x^2(t)}\,dt + \frac{1}{x(0)} \right) = \frac{b}{x(0)}.
\]
Since $\phi_0'(0)=0$, it follows that \(b = 0\), in which case \(\phi_1(s) = a\,x(s)\), which does not vanish if $a\neq 0$. But this contradicts the fact that $\phi_0$ has two zeros on $[-s_r,s_r]$. 

Consequently, the eigenspace associated  Steklov eigenvalue \(\sigma_1\) is actually contained in the space spanned by the functions
\[
\left\{\, \phi_1(s)\cos\theta,\ \phi_1(s)\sin\theta,\ \phi_2(s)\cos 2\theta,\ \phi_2(s)\sin 2\theta \,\right\}.
\]

Furthermore, away from the point \(s = 0\), we can also construct a second linearly independent solution  for e\eqref{eq:equation_for_phi_k} for $k=1$, using Abel’s formula. Given that \(y(s)\) is a known solution, that vanishes only on $s=0$ a second solution can be explicitly find outside the origin, in fact, given $\varepsilon>0$ we have that 
\[
w(s) = y(s) \int_{\varepsilon}^{s} \frac{1}{y^2(t)}\,dt
\]
is a second solution of \eqref{eq:equation_for_phi_k}, for $k=1$. 
Now we verify that \(w(s)\) defines a Steklov eigenfunction. Computing the quotient of its normal derivative and its value at the boundary point \(s = s_r\), we obtain:
\begin{align*}
    \frac{1}{w(s_r)}\frac{\partial w(s_r)}{\partial \nu} 
    &= \frac{1}{\sqrt{\rho(s_r)}} \cdot \frac{w'(s_r)}{w(s_r)} \\
    &= \frac{1}{\sqrt{\rho(s_r)}} \cdot \left( \frac{y'(s_r) \int \frac{1}{y^2(t)}\,dt + \frac{1}{y(s_r)}}{y(s_r) \int \frac{1}{y^2(t)}\,dt} \right) \\
    &= \cot r + \frac{1}{\sqrt{\rho(s_r)}\, y^2(s_r) \int \frac{1}{y^2(t)}\,dt }.
\end{align*}

Note that the normal derivative \(\frac{\partial}{\partial \nu}\) is given by \(\frac{1}{\sqrt{\rho}}\, \partial_s\) at the boundary component \(\{s_r\} \times S^1\), and by \(-\frac{1}{\sqrt{\rho}}\, \partial_s\) at \(\{-s_r\} \times S^1\). Then we have
\[
\frac{\partial  w(\pm s_r)}{\partial \nu} = (\mu + \cot r)\, w(\pm s_r),
\]
where \(\mu = \frac{1}{\sqrt{\rho(s_r)}\, y^2(s_r) \int_{\epsilon}^{s_r} \frac{1}{y^2(t)}\,dt} > 0.\)
This confirms that \(w\) is a Steklov eigenfunction with the eigenvalue \(\sigma = \cot r + \mu\), which is strictly greater than \(\cot r\). Finally, observe that $\phi_1$, as a linear combination $\phi_1(s)=ay(s)+bw(s)$, satisfies
\begin{align*}
    \frac{\partial\phi_1}{\partial \nu}(\pm s_r)&=a\cot{r}  y(\pm s_r)+b(\mu+\cot r) w(\pm s_r),\\
    &= \cot r \phi_1(\pm s_r)+b\mu w(\pm s_r).
\end{align*}
If $b \neq 0$, then we must have $a = 0$. In this case, $\phi_1$ is a Steklov eigenfunction corresponding to an eigenvalue greater than $\cot r$, and consequently cannot be the second eigenfunction. It follows that $\phi_1$ must be a multiple of $y$. In particular, $\sigma_1=\cot R$, as we wanted to show.

\end{proof}

Conversely, if we have a intrinsically rotationally symmetric free boundary immersion of a M\"obius band in \( \Br \) given by first Steklov eigenfunctions,  we can prove that the coordinate functions  satisfy the condition of \textbf{Proposition \ref{prop:Characterization_of_FB_minimal_mobius_band}}. In fact, we have the following converse: 

\begin{prop}\label{thm:FBMMB-Steklov}
Let \(\Phi = (\phi_0, \ldots, \phi_n): (\mathbb{M}, g) \to \mathbb{B}^n(r)\) be a free boundary minimal immersion of a M\"obius band by first Steklov eigenfunctions. Then \((\mathbb{M}, g)\) is intrinsically rotationally symmetric, \(n = 4\), and, up to an isometry of \(\mathbb{B}^n(r)\), the coordinate functions are given by 
\[
\begin{aligned}
\phi_0(s,\theta) &= \psi_0(s), \\
\phi_1(s,\theta) &= \psi_1(s)\cos \theta,\quad\quad \phi_2(s,\theta) = \psi_1(s)\sin \theta, \\
\phi_3(s,\theta) &= \psi_2(s)\cos (2\theta), \quad
\phi_4(s,\theta) = \psi_2(s)\sin (2\theta),
\end{aligned}
\]
where the functions \(\phi_0\), \(\psi_1\), and \(\psi_2\) satisfy the system of ODEs
\begin{equation*}
  \begin{cases}  
      \psi_0'' = -\big(2\psi_1^2 + 8\psi_2^2\big)\psi_0,  &\psi_0(0)=\sqrt{1-a^2}, \psi_0'(0)=0\\
      \psi_1' = \big(1 - 2\psi_1^2- 8\psi_2^2\big) \psi_1,  &\psi_1(0) = 0,\psi_1'(0) = 2a, \\
      \psi_2'' = \big(4 - 2\psi_1^2- 8\psi_2^2\big) \psi_2,  &\psi_2(0) = a,\psi_2'(0) = 0,
  \end{cases}
\end{equation*}
were $0<a<1$, with the following additional properties:
\begin{itemize}
    \item[(i)] The function \(\psi_0(s)\) is even and positive, \(\psi_1(s)\) is odd and vanishes only at \(s = 0\) on \([-s_r, s_r]\), while \(\psi_2(s)\) is even and never vanishes on the interval.
    \item[(ii)] \(\psi_1^2(s) + \psi_2^2(s) \leq \sin^2 r\) for all \(s \in [-s_r, s_r]\).
    \item[(iii)] The boundary conditions hold:
    \begin{itemize}
        \item[(a)] \(\psi_0(-s_r) = \psi_0(s_r) = \cos r\),
        \item[(b)] \(\psi_0'(\pm s_r) = -\sin r \sqrt{\psi_1^2(\pm s_r) + 4\psi_2^2(\pm s_r)}\).
    \end{itemize}
\end{itemize}
\end{prop}
\begin{proof}
First of all, by \textbf{Theorem A}, any free boundary minimal immersion of a Möbius band by first Steklov eigenfunctions is intrinsically rotationally symmetric. Therefore, there exists a positive constant \(s_r > 0\) and a positive function \(\rho : [-s_r, s_r] \times \mathbb{S}^1 \to \mathbb{R}\), which is independent of \(\theta\), such that $g=\rho(s)(ds^2+d\theta^2)$.

As in the proof of Proposition \ref{prop:Characterization_of_FB_minimal_mobius_band}, we apply separation of variables and write the eigenfunctions in the form \(\phi_k(s, \theta) = \psi_k(s)\, s_k(\theta)\), where \(s_k(\theta) = \sin(k\theta)\) or \(s_k(\theta) = \cos(k\theta)\). The radial part \(\psi_k(s)\) then satisfies the second-order differential equation:
\[
\psi_k''(s) = \big(k^2 - 2\rho(s)\big) \psi_k(s).
\] 

We know that the \(\sigma_0\)-eigenfunction can be chosen to be positive. Therefore, the corresponding eigenspace is spanned by \(\{\psi_0(s)\}\), where \(\psi_0\) is a positive solution of the equation
\[
\psi_0''(s) = -2\rho(s)\psi_0(s).
\]
Furthermore, by the Courant's nodal set theorem, the \(\sigma_1\)-eigenfunctions must have exactly two nodal domains. This implies that the eigenspace associated with \(\sigma_1\) is spanned by the set
\[
\{ \psi_0(s),\ \psi_1(s)\cos\theta,\ \psi_1(s)\sin\theta,\ \psi_2(s)\cos 2\theta,\ \psi_2(s)\sin 2\theta \},
\]
where \(\psi_0\) has two zeros in \([-s_r, s_r]\), \(\psi_1\) has exactly one zero in \([-s_r, s_r]\), and \(\psi_2\) does not vanish in the interval. Moreover, as observed in the proof of Proposition \ref{prop:Characterization_of_FB_minimal_mobius_band}, the parity properties of \(\psi_0(s)\) show that we cannot construct an even function satisfying the required properties of having two zeros. Hence, we may reduce the basis of this eigenspace to:
\[
\{ \psi_1(s)\cos\theta,\ \psi_1(s)\sin\theta,\ \psi_2(s)\cos 2\theta,\ \psi_2(s)\sin 2\theta \}.
\]

As a result, the immersion is determined by functions spanning the eigen\-spaces corresponding to the first two Steklov eigenvalues. That is, the space is spanned by:
\[
\{ \psi_0(s),\ \psi_1(s)\cos\theta,\ \psi_1(s)\sin\theta,\ \psi_2(s)\cos 2\theta,\ \psi_2(s)\sin 2\theta \},
\]
where:
\begin{itemize}
    \item \(\psi_0(s)\) is even and positive,
    \item \(\psi_1(s)\) is odd and has exactly one zero in \([-s_r, s_r]\),
    \item \(\psi_2(s)\) is even and positive. 
\end{itemize}

We conclude that \(n \leq 4\). Moreover, as shown by \cite{Carlos}\footnote{The author of \cite{Carlos}, in personal communication, explained to us that the conclusion of his work also holds in geodesic balls of three-dimensional space forms. This generalization is due to an idea suggested by Albert Cerezo.}, there exists no free boundary minimal Möbius band in \(\mathbb{B}^3(r)\). Therefore, it must be the case that \(n = 4\).

Furthermore, we can compose the immersion \(\Phi\) with an isometry \(R\) of \(\mathbb{B}^4(r)\) so that:
\[
R \circ \Phi = \big( \psi_0(s),\ \psi_1(s)\cos\theta,\ \psi_1(s)\sin\theta,\ \psi_2(s)\cos 2\theta,\ \psi_2(s)\sin 2\theta \big).
\]

Finally, since the immersion is isometric, we have
\[
\rho(s) = \left\| \frac{\partial \Phi}{\partial s} \right\|^2 = \left\| \frac{\partial \Phi}{\partial \theta} \right\|^2,
\]
from which it follows that:
\[
\rho(s) = \psi_1^2(s) + 4\psi_2^2(s) = \left(\psi_0'(s)\right)^2+\left(\psi_1'(s)\right)^2+\left(\psi_2'(s)\right)^2.
\]

Moreover, since the immersion descends to the Möbius band (i.e. it is invariant under the quotient that defines the Möbius topology), the coordinate functions must satisfy the following symmetries:
\begin{align*}
    \psi_0(s) &= \phi_0(s, \theta) = \phi_0(-s, \theta + \pi) = \psi_0(-s) \quad \Rightarrow \quad \psi_0 \text{ is even}, \\
    \psi_1(s)\cos\theta &= \phi_1(s, \theta) = \phi_1(-s, \theta + \pi) = \psi_1(-s)\cos(\theta + \pi) = -\psi_1(-s)\cos\theta \\
    &\Rightarrow\quad \psi_1(s) = -\psi_1(-s) \quad \text{(odd)}, \\
    \psi_2(s)\cos 2\theta &= \phi_2(s, \theta) = \phi_2(-s, \theta + \pi) = \psi_2(-s)\cos(2\theta + 2\pi) = \psi_2(-s)\cos 2\theta \\
    &\Rightarrow\quad \psi_2(s) = \psi_2(-s) \quad \text{(even)}.
\end{align*}

Thus, the coordinate functions \(\psi_0, \psi_1, \psi_2\) satisfy the conditions $(i)$, $(ii)$, and $(iii)$ in \textbf{Proposition~\ref{prop:Characterization_of_FB_minimal_mobius_band}}.
\end{proof}

\subsection{Solving the ODE}

Given \( 0 < r \leq \frac{\pi}{2} \), we will prove that there exists an initial condition \( a \in (0,1) \) for the system of equations \eqref{eq:System_of_ODEs_FBMMB} such that all conditions in \textbf{Proposition~\ref{prop:Characterization_of_FB_minimal_mobius_band}} are satisfied.

This system of equations was studied in \cite{ElSoufi_Giacomini_Jazar} and \cite{Jakobson_Nadirashvili_Polterovich_2006}, where the authors proved the existence and uniqueness of a minimal Klein bottle in \(\mathbb{S}^4\) immersed by first eigenfunctions of the Laplacian that maximizes the first normalized eigenvalue

 Observe that the system \eqref{eq:System_of_ODEs_FBMMB} is invariant under the transformation,
\[(y(s), z(s)) \mapsto (-y(-s), z(-s)), \] 
which implies that \( y \) is an odd function and \( z\) 
is an even function. Consequently, condition $(i)$ in \textbf{Proposition\ref{prop:Characterization_of_FB_minimal_mobius_band}} is always satisfied.

El Soufi, Giacomini, and Jazar established in \cite{ElSoufi_Giacomini_Jazar} that this system admits two independent first integrals. Here, a first integral is a function \( H(y, z, y', z',a) \) that satisfies the equation
\[
\dfrac{d}{ds} H(y, z, y', z',a) = y'\frac{\partial H}{\partial y} + z'\frac{\partial H}{\partial z} + y''\frac{\partial H}{\partial y'} + z''\frac{\partial H}{\partial z'} = 0.
\]
for all solutions. In fact, the explicit functions
\begin{equation}\label{eq:first_integral}
\begin{cases}
H_1(y, z, y', z',a) &= (y^2 + 4z^2)^2 - y^2 - 16z^2 + (y')^2 + 4(z')^2 + 4a^2(3 - 4a^2), \\[5pt]
H_2(y, z, y', z',a) &= 12z^2(z^2 - 1) + 3y^2 z^2 + z^2 (y')^2 - 2yy' zz' \\ 
&\quad + (3 + y^2)(z')^2 + 4a^2(3 - 4a^2),
\end{cases}
\end{equation}
are the two independent first integrals of the system \eqref{eq:System_of_ODEs_FBMMB}. The orbits of the solution are therefore contained in the algebraic variety
\[
\begin{cases}
H_1(y, z, y', z',a) = 0, \\ 
H_2(y, z, y', z',a) = 0,
\end{cases}
\]
in the phase space. In particular, in Section 3.4 of \cite{ElSoufi_Giacomini_Jazar}  they show that  for every $0\le a\le 1$ the orbits lie in the region 
\[
y^2 + z^2 \leq 1.
\]

Since the orbits lie in the region \(y^2+z^2\le 1\) the solutions are bounded. This allows us to define the function
\[
x(s) := \sqrt{1 - y^2(s) - z^2(s)}.
\]
Note that \( x^2 + y^2 + z^2 = 1 \). Differentiating this identity yields
\[
x' x = -y' y - z' z \quad \Rightarrow \quad (x' x)^2 = (y' y)^2 + (z' z)^2 + 2 y' y z' z.
\]
Using the second equation in \eqref{eq:first_integral}, we obtain
\begin{align*}
(x' x)^2 &= (y' y)^2 + (z' z)^2 + 12 z^2(z^2 - 1) + 3 y^2 z^2 + (z y')^2 \\
&\quad + (3 + y^2)(z')^2 + 4 a^2(3 - 4 a^2), \\
&= (y^2 + z^2)(y')^2 + (3 + y^2 + z^2)(z')^2 + 12 z^2(z^2 - 1) \\
&\quad + 3 y^2 z^2 + 4 a^2(3 - 4 a^2), \\
&= (1 - x^2)(y')^2 + (4 - x^2)(z')^2 + 12 z^2(z^2 - 1) \\
&\quad + 3 y^2 z^2 + 4 a^2(3 - 4 a^2).
\end{align*}

Next, we use the first equation in \eqref{eq:first_integral}:
\begin{align*}
x^2 \left( (x')^2 + (y')^2 + (z')^2 \right) &= (y')^2 + 4 (z')^2 + 12 z^2(z^2 - 1) + 3 y^2 z^2 \\
&\quad + 4 a^2(3 - 4 a^2), \\
&= - (y^2 + 4 z^2)^2 + y^2 + 16 z^2 - 4 a^2(3 - 4 a^2) \\
&\quad + 12 z^2(z^2 - 1) + 3 y^2 z^2 + 4 a^2(3 - 4 a^2), \\
&= -y^4 - 4 z^4 - 5 y^2 z^2 + y^2 + 4 z^2, \\
&= (1 - y^2 - z^2)(y^2 + 4 z^2), \\
&= x^2 (y^2 + 4 z^2).
\end{align*}

Therefore, if \( x \neq 0 \), we conclude that
\begin{equation} \label{eq: y^2+4z^2=x'^2+y'^2+z'^2}
(x')^2 + (y')^2 + (z')^2 = y^2 + 4 z^2.
\end{equation}

Furthermore, observe that  we prove in \eqref{eq:ode_for_x} that

\begin{equation} \label{eq:_equacao_para_x(s)}
x'' = - 2(y^2 + 4 z^2) x.
\end{equation}
In particular, we are interested in the solutions that remain in the upper hemisphere, that is,
\[
x(s) \geq 0.
\]

Our goal is to find an initial condition \( a \in [0,1) \) such that the system satisfies the condition stated in \textbf{Proposition~\ref{prop:Characterization_of_FB_minimal_mobius_band}}.

To this end, we define the function
\[
F : [0,1) \times \left[0, \frac{\pi}{2}\right] \times \mathbb{R}_{\geq 0} \to \mathbb{R}^2
\]
by
\[
F(a, r, s) = \left( \cos r - x_a(s),\; \sin r \cdot \sqrt{y_a^2(s) + 4 z_a^2(s)} + x'_a(s) \right),
\]
where \( y_a \) and \( z_a \) denote the unique solutions to the system \eqref{eq:System_of_ODEs_FBMMB}, and \( x_a(s) = \sqrt{1 - y_a^2(s) - z_a^2(s)} \).

First, for \( a = \sqrt{\frac{3}{8}} \), Jakobson, Nadirashvili, and Polterovich~\cite{Jakobson_Nadirashvili_Polterovich_2006} proved that the corresponding solutions are periodic and that \( x(s) \) has two zeros within one period. They used this solution to construct a minimal Klein bottle in \( \mathbb{S}^4 \).

We will show that half of this surface corresponds to a free boundary minimal Möbius band contained in the upper hemisphere \( \mathbb{B}^4 \left( \frac{\pi}{2} \right) \). Moreover, we will prove that for each \( r \in \left(0, \frac{\pi}{2} \right ] \), there exists an \( a \in (0,1) \) and a corresponding \( s_r > 0 \) such that
\[
F(a, r, s_r) = (0, 0).
\]

\begin{lema}\label{lema:FBM_Mobius_band_in_spherical-cap}
Given \( r \in \left(0, \frac{\pi}{2} \right) \), there exist \( a \in (0,1) \) and \( s_r > 0 \) such that:
\begin{enumerate}
    \item[$(i)$] For all \( s \in (-s_{r}, s_{r}) \), we have \( x_{a}(s) > \cos r \), and 
    \[
    x_{a}(\pm s_{r}) = \cos{r}.
    \]
    
    \item[$(ii)$] 
    \[
    F\left(a, r, s_{r}\right) = (0, 0).
    \]
    \item[$(iii)$] \( z(s) > 0 \) for all \( s \in [-s_r, s_r] \), and \( y(s) = 0 \) if and only if \( s = 0 \).

    \item[$(iv)$] 
    As \( r \to \frac{\pi}{2} \), we have \( a \to \sqrt{\frac{3}{8}} \), and as \( r \to 0 \), we have \( a \to 0 \).
\end{enumerate}
\end{lema}
\begin{proof}
As the system \eqref{eq:System_of_ODEs_FBMMB} admits two first integrals, it was observed in \textbf{Section~3.2} of~\cite{ElSoufi_Giacomini_Jazar} that, for any \( a \), the curve
\[
a^2 y^2 - (3 - 4a^2) z^2 + a^2(3 - 4a^2) = 0
\]
contains an orbit of a particular solution of the system.

We can use this relation to decouple the equations and find a value of \( a \) for which the required conditions are satisfied. In fact, as we also have that \( x^2+y^2+z^2=1\), from these two vinculum equations we obtain

\begin{align*}
    y^2&= \frac{3-4a^2}{3}-\frac{3-4a^2}{3-3a^2} x^2,\\
    z^2&= \frac{4a^2}{3}-\frac{a^2}{3-3a^2}x^2,
\end{align*}
Therefore, the conformal factor \( \rho = y^2 + 4z^2 \) can be expressed in terms of the function \( x \). In fact, we obtain
\[
\rho = 1 + 4a^2 - \frac{1}{1 - a^2} x^2.
\]
Thus, we can rewrite the condition \( F(a, r, s_r) = 0 \) as
\begin{equation}\label{eq:New_FB_condition_for_the_sitem_of_ode}
\begin{split}
    x_a(s_r) &= \cos r, \\
(x'_a(s_r))^2 &= \sin^2 r \cdot \left(1 + 4a^2 - \frac{1}{1 - a^2} x_a^2(s_r)\right).
\end{split}
\end{equation}
Now, this condition involves only the function \( x \), and we know that \( x \) satisfies equation~\eqref{eq:_equacao_para_x(s)}. Thus, we obtain
\begin{equation}\label{eq:new_equacao_para_x}
\begin{split}
    x''(s) &= -2\left(1 + 4a^2 - \frac{1}{1 - a^2} x^2(s)\right) x(s), \\
    x(0) &= \sqrt{1 - a^2}, \quad x'(0) = 0.
\end{split}
\end{equation}
Observe that since \( x(0) = \sqrt{1 - a^2} > 0 \), there exists \( \tilde{s} > 0 \) such that \( x(s) \geq 0 \) for all \( s \in [0, \tilde{s}] \). Moreover, \( g = \rho (ds^2 + d\theta^2) \) is a Riemannian metric on the cylinder \( [-\tilde{s}, \tilde{s}] \times \mathbb{S}^1 \), because $\rho>0$.

Consider the even extension of \( x \) to the interval \( [-\tilde{s}, \tilde{s}] \). Then equation~\eqref{eq:new_equacao_para_x} is equivalent to 
\[
 \Delta_g x + 2x = 0.
\]
By the Hopf's Lemma, it follows that
\[
\frac{\partial x}{\partial s}(s) < 0
\quad \text{for all } s > 0 \text{ such that } x(s) \ge 0.
\]
 Moreover, by the Sturm Comparison Theorem, we can also show that \( x \) must have a zero. In fact, we compare equation \eqref{eq:new_equacao_para_x} with the linear equation
\begin{equation}\label{eq:equacao_de_comparacao}
    w'' = -2(1 + 4a^2) w.
\end{equation}
Then, by Sturm's  Theorem, between any two consecutive zeros of a solution to~\eqref{eq:equacao_de_comparacao}, there exists at least one zero of \( x \).

Thus, \( x(s) \) is strictly decreasing while it remains positive, and there exists \( s_a > 0 \) such that \( x(s_a) = 0 \). 

Given \( r > 0 \), let \( s_r \) denote the first positive solution of the equation
\[
x(s) - \cos r = 0.
\]
Then, for all \( s \in (-s_r, s_r) \), we have \( x(s) > \cos r > 0 \), and
\[
x(\pm s_r) = \cos r,
\]
where \( x \) is a solution of equation~\eqref{eq:new_equacao_para_x}.
It remains to show that we can find a value of \( a \in (0,1) \) such that the second equation in~\eqref{eq:New_FB_condition_for_the_sitem_of_ode} is satisfied. 

Multiplying both sides of equation~\eqref{eq:new_equacao_para_x} by \( x'(s) \), we reduce it to a first-order equation:
\[
(x'(s))^2 = - (2 + 8a^2)\, x^2(s) + \frac{x^4(s)}{1 - a^2} + (1 - a^2)(8a^2 + 1).
\]

We seek \( a \in (0,1) \) such that
\[
- (2 + 8a^2)\, x^2(s_r) + \frac{x^4(s_r)}{1 - a^2} + (1 - a^2)(8a^2 + 1) 
= \sin^2 r \cdot \left(1 + 4a^2 - \frac{1}{1 - a^2} x^2(s_r)\right),
\]
where \( x^2(s_r) = \cos^2 r \) and \( r \in (0, \frac{\pi}{2}) \) is fixed. This leads to the following rational equation in \( a \):

\begin{align*}
0 =& -(2 + 8a^2)(1 - a^2)\cos^2 r + \cos^4 r + (1 - a^2)^2(8a^2 + 1) \\
& - \sin^2 r (1 + 4a^2)(1 - a^2) + \sin^2 r \cos^2 r.
\end{align*}

After simplification, this becomes:
\[
0 = a^2(4a^2 - 5)\cos^2 r + a^2(1 - a^2)(3 - 8a^2).
\]

Since \( a \neq 0 \), we are left with the polynomial equation:
\begin{equation}\label{eq:polonomia_eq_for_a}
    (4a^2 - 5)\cos^2 r + (1 - a^2)(3 - 8a^2) = 0.
\end{equation}

It is straightforward to verify that for each \( r \in (0, \frac{\pi}{2}) \), this equation has exactly one root \( a \in (0,1) \), and this root satifies $\sqrt{1-a^2}< \cos r.
$. Moreover, we observe that:
As \( r \to 0 \), we have \( a \to 0 \),
as \( r \to \frac{\pi}{2} \), we have \( a \to \sqrt{\frac{3}{8}} \).  

Observe that if \( z(s) = 0 \), then
\[
y^2(s) = 4a^2 - 3 < 0, \quad \text{since } a < \sqrt{\frac{3}{8}}.
\]
Hence, \( z(s) \) never vanishes. Furthermore, if \( y(s) = 0 \), we obtain
\[
x^2(s) = 1 - a^2 = x^2(0).
\]
However, since \( x(s) \) is strictly decreasing, it follows that the only point where \( y(s) = 0 \) occurs is at \( s = 0 \).

\end{proof}

\begin{rmk}\label{rmk:r=pi/2}
 When $r = \frac{\pi}{2}$, the solution of \eqref{eq:polonomia_eq_for_a} is $a = \sqrt{\frac{3}{8}}$, which yields a free boundary minimal Möbius band in the hemisphere. Notably, this particular parameter value coincides with that appearing in the Klein bottle construction of \cite{ElSoufi_Giacomini_Jazar}; see \textbf{Theorem 3.1}.
\end{rmk}

\begin{thmB}\label{thm:TheoremB}
Let \( 0 < r < \pi/2 \). Then there exists a free boundary minimal immersion of a Möbius band by first steklov eigenfunction
\[
\Phi = (\phi_0, \dots, \phi_4) : (\mathbb{M}, g) \to \mathbb{B}^4(r).
\]
\end{thmB}
\begin{proof}
  By \textbf{Lemma \ref{lema:FBM_Mobius_band_in_spherical-cap}}, we can construct solutions of equation \eqref{eq:minimal_condition}, for a specific choice of parameter $a=a_r\in (0,1)$ and on a specific interval $[-s_r,s_r]$ with $s_r>0$, that satisfy all the conditions of \textbf{Proposition~\ref{prop:Characterization_of_FB_minimal_mobius_band}}. This, in turn, allows us to construct a free boundary minimal Möbius band for every \(0 < r < \pi/2\).
\end{proof}

\printbibliography

@misc{LM23,
      title={Eigenvalue problems and free boundary minimal surfaces in spherical caps}, 
      author={V. Lima and A. Menezes},
      year={2023},
      eprint={2307.13556},
      archivePrefix={arXiv},
}

@article {FS2011,
    AUTHOR = {Fraser, Ailana and Schoen, Richard},
     TITLE = {The first {S}teklov eigenvalue, conformal geometry, and
              minimal surfaces},
   JOURNAL = {Adv. Math.},
    VOLUME = {226},
      YEAR = {2011},
    NUMBER = {5},
     PAGES = {4011--4030},
      ISSN = {0001-8708,1090-2082},
}

@article{FS2012,
    AUTHOR = {Fraser, Ailana and Schoen, Richard},
     TITLE = {Sharp eigenvalue bounds and minimal surfaces in the ball},
   JOURNAL = {Invent. Math.},
    VOLUME = {203},
      YEAR = {2016},
    NUMBER = {3},
     PAGES = {823--890},
}

@article {medvedev2025,
    AUTHOR = {Medvedev, Vladimir},
     TITLE = {On free boundary minimal submanifolds in geodesic balls in
              {$H^n$} and {$ S^n_+$}},
   JOURNAL = {Math. Z.},
    VOLUME = {310},
      YEAR = {2025},
    NUMBER = {1},
     PAGES = {Paper No. 10, 32},
}

@article{FraserSargent,
author = {Fraser, Ailana and Sargent, Pam},
year = {2020},
title = {Existence and Classification of $\pmb {\mathbb {S}}^1$-Invariant Free Boundary Minimal Annuli and Möbius Bands in $\pmb {\mathbb {B}}^n$ },
volume = {31},
journal = {The Journal of Geometric Analysis}
}

@article {surveySteklov,
    AUTHOR = {Colbois, Bruno and Girouard, Alexandre and Gordon, Carolyn and
              Sher, David},
     TITLE = {Some recent developments on the {S}teklov eigenvalue problem},
   JOURNAL = {Rev. Mat. Complut.},
    VOLUME = {37},
      YEAR = {2024},
    NUMBER = {1},
     PAGES = {1--161},
}

@misc{KKMS,
      title={Embedded minimal surfaces in $\mathbb{S}^3$ and $\mathbb{B}^3$ via equivariant eigenvalue optimization}, 
      author={Karpukhin, Mikhail and Kusner, Robert and McGrath, Peter and Ster, Daniel},
      year={2024},
      eprint={2402.13121},
      archivePrefix={arXiv},
      primaryClass={math.DG},
}

@article{ArendtMazzeo,
    AUTHOR = {Arendt, Wolfgang and Mazzeo, Rafe},
     TITLE = {Friedlander's eigenvalue inequalities and the
              {D}irichlet-to-{N}eumann semigroup},
   JOURNAL = {Commun. Pure Appl. Anal.},
    VOLUME = {11},
      YEAR = {2012},
    NUMBER = {6},
     PAGES = {2201--2212},
}

@inproceedings {Lisurvey,
    AUTHOR = {Li, Martin Man-chun},
     TITLE = {Free boundary minimal surfaces in the unit ball: recent
              advances and open questions},
 BOOKTITLE = {Proceedings of the {I}nternational {C}onsortium of {C}hinese
              {M}athematicians 2017},
     PAGES = {401--435},
 PUBLISHER = {Int. Press, Boston, MA},
      YEAR = {2020},
}

@article{Jakobson_Nadirashvili_Polterovich_2006,
    AUTHOR = {Jakobson, Dmitry and Nadirashvili, Nikolai and Polterovich, Iosif},
     TITLE = {Extremal metric for the first eigenvalue on a {K}lein bottle},
   JOURNAL = {Canad. J. Math.},
    VOLUME = {58},
      YEAR = {2006},
    NUMBER = {2},
     PAGES = {381--400},
}

@article{Takahashi66,
    AUTHOR = {Takahashi, Tsunero},
     TITLE = {Minimal immersions of {R}iemannian manifolds},
   JOURNAL = {J. Math. Soc. Japan},
    VOLUME = {18},
      YEAR = {1966},
     PAGES = {380--385},
}

@article{Simons68,
    AUTHOR = {Simons, James},
     TITLE = {Minimal varieties in {R}iemannian manifolds},
   JOURNAL = {Ann. of Math. (2)},
    VOLUME = {88},
      YEAR = {1968},
     PAGES = {62--105},
}

@article{ElSoufi_Giacomini_Jazar,
    AUTHOR = {El Soufi, Ahmad and Giacomini, Hector and Jazar, Mustapha},
     TITLE = {A unique extremal metric for the least eigenvalue of the {L}aplacian on the {K}lein bottle},
   JOURNAL = {Duke Math. J.},
    VOLUME = {135},
      YEAR = {2006},
    NUMBER = {1},
     PAGES = {181--202},
}

@article{Obata,
    author = {Obata, Morio},
     title = {Certain conditions for a {R}iemannian manifold to be isometric with a sphere},
   journal = {J. Math. Soc. Japan},
    volume = {14},
      year = {1962},
     pages = {333--340},
}

@article{Carlos,
author = {Cardona, Carlos},
year = {2024},
month = {12},
title = {Non-existence of free boundary minimal Möbius bands in the unit three-ball},
journal = {Proceedings of the American Mathematical Society},
}

@article{Petrides,
author = {Romain Petrides},
title = {{Maximizing Steklov eigenvalues on surfaces}},
volume = {113},
journal = {Journal of Differential Geometry},
number = {1},
pages = {95 -- 188},
year = {2019},
}
\end{document}